\documentclass[10pt]{amsart}

\usepackage{latexsym,exscale,enumerate,amsfonts,amssymb, ulem, xparse, mathtools}
\usepackage{amsmath,amsthm,amsfonts,amssymb,amscd, stmaryrd,textcomp}
\usepackage{hyperref}
\usepackage{ulem}
\usepackage{bbold}
\usepackage[usenames,dvipsnames]{xcolor}
\colorlet{green}{black!30!green} 

\usepackage{tikz}
\usetikzlibrary{calc}
\usetikzlibrary{arrows}
\usetikzlibrary{decorations.markings}
\usetikzlibrary{decorations.pathreplacing}
\usetikzlibrary{arrows,shapes,positioning}
\tikzstyle directed=[postaction={decorate,decoration={markings,
    mark=at position #1 with {\arrow{>}}}}]
\tikzstyle rdirected=[postaction={decorate,decoration={markings,
    mark=at position #1 with {\arrow{<}}}}]
\tikzset{anchorbase/.style={baseline={([yshift=-0.5ex]current bounding box.center)}}}

\tikzset{
    partial ellipse/.style args={#1:#2:#3}{
        insert path={+ (#1:#3) arc (#1:#2:#3)}
    }
}

\newcommand{\R}{\mathbb{R}}
\newcommand{\C}{\mathbb{C}}
\newcommand{\Z}{\mathbb{Z}}
\newcommand{\N}{\mathbb{N}}
\newcommand{\D}{\mathbb{D}}

\newcommand{\B}{{\bf B}}

\newcommand{\calC}{\mathcal{C}}
\newcommand{\calP}{\mathcal{P}}

\newcommand{\id}{\mathrm{id}}

\DeclareMathOperator{\Hom}{Hom}
\DeclareMathOperator{\End}{End}
\DeclareMathOperator{\HOM}{HOM}
\DeclareMathOperator{\END}{END}
\DeclareMathOperator{\Aut}{Aut}
\DeclareMathOperator{\qdim}{qdim}
\DeclareMathOperator{\Kom}{Kom}

\renewcommand{\tilde}{\widetilde}

\theoremstyle{theorem}
\newtheorem{theorem}{Theorem}
\newtheorem{lemma}{Lemma}
\newtheorem{corollary}{Corollary}
\newtheorem{remark}{Remark}
\newtheorem{definition}{Definition}
\newtheorem{example}{Example}
\newtheorem{problem}{Problem}

\newtheorem{proposition}{Proposition}

\begin{document}
%

\title{Khovanov-Seidel braid representation}

\author{Hoel Queffelec}
\address{IMAG\\ Univ. Montpellier\\ CNRS \\ Montpellier \\ France}
\address{Mathematical Sciences Institute \\ Australian National University}
\email{hoel.queffelec@umontpellier.fr}

\begin{abstract}
  These are lecture notes from a lecture series given at CIRM in the Fall 2023. They give a down-to-earth introduction to Khovanov and Seidel's categorical representation of Artin-Tits groups, emphasizing the fact that it is all explicitly computable. Several prospective applications to (geometric) group theory are mentioned.
\end{abstract}

\maketitle

\section{Introduction}

These lecture notes are devoted to a categorical representation of braid groups (and then some Artin-Tits groups) that was introduced by Khovanov and Seidel in the early 2000's~\cite{KhS}. Khovanov and Seidel's categorical action lifts up the more classical Burau representation~\cite{Burau}. This latter matrix representation has been proven to be faithful in the case of the 3-strand braid group (see~\cite{Birman_book} and references therein), and unfaithful for 5 strands or more~\cite{Moody, LongPaton, Bigelow_Burau}. Which leaves the 4-strand case widely open, and one of the most tantalizing questions in the field.

Once upgraded to a categorical version, the representation becomes faithful: that's the main result of Khovanov and Seidel, that was later extended to more general classes of Artin-Tits groups~\cite{RouquierZimmerman,KhS,BravThomas,Riche, IshiiUedaUehara,GTW}. However, as for most questions on Artin-Tits groups (see~\cite{GodelleParis}), most cases remain open, which is closely related to the question of the faithfulness of the action of Artin-Tits groups on Soergel bimodules~\cite{Rouquier_conjecture,Jensen}. Interestingly, the faithfulness issue of the Burau representation was solved at the same time in two different manners: by Khovanov and Seidel, as explained, by going up one categorical level, but also by Lawrence, Krammer and Bigelow~\cite{Lawrence, Krammer, Bigelow_linearity} who deformed the second symmetric power of the Burau representation to make it faithful, thus proving braid groups to be linear.

In these lecture notes, I want to sell Khovanov-Seidel's machinery as a very handy and highly computable process. The categorical objects involved are rather elementary, and the computations come down to easy linear algebra. But at the same time, it seems to carry quite a lot of geometry -- although exploiting it turns out to be harder than one might expect. In that regard, I'll highlight a few questions originating from geometric group theory where I'm hoping to get some insight from Khovanov-Seidel categorical action.

\noindent {\bf Acknowledgements:} these notes are based on four lectures given at CIRM in November 2023. I'd like to thank Claire Amiot, Thomas Brüstle, Yann Palu and Pierre-Guy Plamondon for their invitation and encouragements, as well as the audience (in particular Agnès Gadbled, Edmnd Heng, Yu Qiu, Anne-Laure Thiel and Emmanuel Wagner). I've learnt all of this material from discussions with Asilata Bapat, Anand Deopurkar and Tony Licata, and all the geometric group theory aspects come from Thomas Haettel's infinite patience. 

This work has been funded by the European Union's Horizon 2020 research and innovation programme under the Marie Sklodowska-Curie grant agreement No 101064705.

\section{Artin-Tits groups and representations}

\subsection{Braids and Artin-Tits groups}

\begin{definition}
  A \emph{Coxeter diagram}, or \emph{Coxeter graph}, is a graph with edges labeled in the set $\{3,\dots,\infty\}$. The label of $3$-labeled edges is usually omitted. 
\end{definition}

Coxeter graphs represent Coxeter systems, that generate Coxeter groups. These are sets $(S,R)$ of generators $s_i$ of order $2$ subject to relations $r_k$ of the kind: $(s_is_j)^{m_k}=1$. The vertices of the graph correspond to the generators, while edge labeling indicates the $m_k$'s. Non-adjacent vertices correspond to $m_k=2$ (then the generators commute), while $m_k=\infty$ means that there is no relation between the corresponding generators. Such a graph can be assumed to have at most one edge between any two vertices.

\begin{example} \label{ex:S3}
  \[
    A_3=\begin{tikzpicture}[anchorbase]
      \node (1) at (1,0) {$\bullet$};
      \node (2) at (2,0) {$\bullet$};
      \node (3) at (3,0) {$\bullet$};
      \draw (1.center) -- (2.center);
      \draw (2.center) -- (3.center);
    \end{tikzpicture}
  \]
  The Coxeter group in type $A_3$ is the symmetric group on four letters $\mathfrak{S}_4$.
\end{example}

\begin{example}
  \[
  \begin{tikzpicture}[anchorbase]
    \node (1) at (1,2) {$\bullet$};
    \node [above] at (1) {$1$};
    \node (2) at (0,1) {$\bullet$};
    \node [left] at (2) {$2$};
    \node (3) at (2,1) {$\bullet$};
    \node [right] at (3) {$3$};
    \node (4) at (1,0) {$\bullet$};
    \node [below] at (4) {$4$};
    \draw (1.center) -- node [midway,above,rotate=45] {$\infty$} (2.center);
    \draw (1.center) -- (3.center);
    \draw (3.center) -- node [midway, above, rotate=45] {$4$} (4.center);
  \end{tikzpicture}
  \quad  \leftrightarrow
  \quad
  \begin{cases}
    S=\{s_1,s_2,s_3,s_4\}
    \\
    R= \begin{Bmatrix}   s_1s_3 s_1 = s_3 s_1 s_3,\; s_3s_4s_3s_4=s_4s_3s_4s_3 \\
     s_2s_3=s_3s_2,\; s_1s_4=s_4s_1
\end{Bmatrix}
    \end{cases}
  \]
  Then:
  \[
G_\Gamma=\frac{\langle s_1,s_2,s_3,s_4 \rangle}{\{s_i^2=1\}\cup R}
  \]
\end{example}

\begin{example}
 Generalizing Example~\ref{ex:S3}, Weyl groups are Coxeter groups.
\end{example}

\begin{theorem}
  A Coxeter group is finite if and only if its associated Coxeter diagram is of type: $A_n$, $B_n=C_n$, $D_n$, $E_6$, $E_7$, $E_8$, $F_4$, $G_2$, $H_2$, $H_3$, $H_4$, $I_n$.
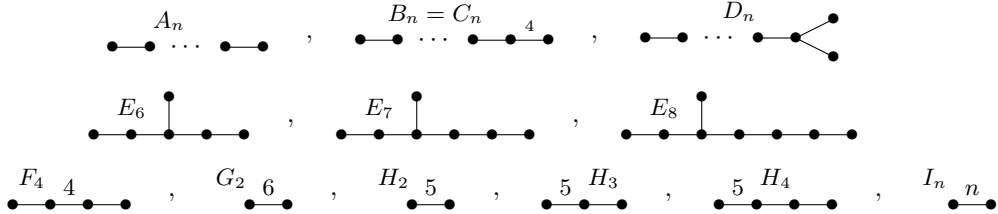
  \begin{gather*}
  \begin{tikzpicture}[anchorbase,scale=.5]
      \node at (2.5,.7) {\small $A_n$};
      \node (1) at (1,0) {$\bullet$};
      \node (2) at (2,0) {$\bullet$};
      \node (d) at (3,0) {$\cdots$};
      \node (n-1) at (4,0) {$\bullet$};
      \node (n) at (5,0) {$\bullet$};
      \draw (1.center) -- (2.center);
      \draw (n-1.center) -- (n.center);
    \end{tikzpicture}
    \quad,\quad
    \begin{tikzpicture}[anchorbase,scale=.5]
            \node at (3,.7) {\small $B_n=C_n$};
      \node (1) at (1,0) {$\bullet$};
      \node (2) at (2,0) {$\bullet$};
      \node (d) at (3,0) {$\cdots$};
      \node (n-2) at (4,0) {$\bullet$};
      \node (n-1) at (5,0) {$\bullet$};
      \node (n) at (6,0) {$\bullet$};
      \draw (1.center) -- (2.center);
      \draw (n-2.center)  -- (n-1.center);
      \draw (n-1.center) -- (n.center) node [midway,above]{\tiny $4$} node [midway,below] {\tiny \vphantom{$4$}};
    \end{tikzpicture}
    \quad,\quad
    \begin{tikzpicture}[anchorbase,scale=.5]
      \node at (3.5,.7) {\small $D_n$};
      \node (1) at (1,0) {$\bullet$};
      \node (2) at (2,0) {$\bullet$};
      \node (d) at (3,0) {$\cdots$};
      \node (n-3) at (4,0) {$\bullet$};
      \node (n-2) at (5,0) {$\bullet$};
      \node (n-1) at (6,.5) {$\bullet$};
      \node (n) at (6,-.5) {$\bullet$};
      \draw (1.center) -- (2.center);
      \draw (n-3.center) -- (n-2.center);
      \draw (n-2.center) -- (n-1.center);      
      \draw (n-2.center) -- (n.center);
      \end{tikzpicture}
\\
    \begin{tikzpicture}[anchorbase,scale=.5]
      \node at (2,.7) {\small $E_6$};
      \node (1) at (1,0) {$\bullet$};
      \node (2) at (2,0) {$\bullet$};
      \node (3) at (3,0) {$\bullet$};
      \node (4) at (3,1) {$\bullet$};
      \node (5) at (4,0) {$\bullet$};
      \node (6) at (5,0) {$\bullet$};
      \draw (1.center) -- (2.center);
      \draw (2.center) -- (3.center);
      \draw (3.center) -- (4.center);      
      \draw (3.center) -- (5.center);
      \draw (5.center) -- (6.center);
      \end{tikzpicture}
    \quad,\quad
    \begin{tikzpicture}[anchorbase,scale=.5]
      \node at (2,.7) {\small $E_7$};
      \node (1) at (1,0) {$\bullet$};
      \node (2) at (2,0) {$\bullet$};
      \node (3) at (3,0) {$\bullet$};
      \node (4) at (3,1) {$\bullet$};
      \node (5) at (4,0) {$\bullet$};
      \node (6) at (5,0) {$\bullet$};
      \node (7) at (6,0) {$\bullet$};
      \draw (1.center) -- (2.center);
      \draw (2.center) -- (3.center);
      \draw (3.center) -- (4.center);      
      \draw (3.center) -- (5.center);
      \draw (5.center) -- (6.center);
      \draw (6.center) -- (7.center);
      \end{tikzpicture}
    \quad,\quad
    \begin{tikzpicture}[anchorbase,scale=.5]
      \node at (2,.7) {\small $E_8$};
      \node (1) at (1,0) {$\bullet$};
      \node (2) at (2,0) {$\bullet$};
      \node (3) at (3,0) {$\bullet$};
      \node (4) at (3,1) {$\bullet$};
      \node (5) at (4,0) {$\bullet$};
      \node (6) at (5,0) {$\bullet$};
      \node (7) at (6,0) {$\bullet$};
      \node (8) at (7,0) {$\bullet$};
      \draw (1.center) -- (2.center);
      \draw (2.center) -- (3.center);
      \draw (3.center) -- (4.center);      
      \draw (3.center) -- (5.center);
      \draw (5.center) -- (6.center);
      \draw (6.center) -- (7.center);
      \draw (7.center) -- (8.center);
    \end{tikzpicture}
    \\
    \begin{tikzpicture}[anchorbase,scale=.5]
      \node at (1.5,.7) {\small $F_4$};
      \node (1) at (1,0) {$\bullet$};
      \node (2) at (2,0) {$\bullet$};
      \node (3) at (3,0) {$\bullet$};
      \node (4) at (4,0) {$\bullet$};
      \draw (1.center) -- (2.center);
      \draw (2.center) -- (3.center) node [midway, above] {\small $4$};
      \draw (3.center) -- (4.center);      
    \end{tikzpicture}
    \quad,\quad
    \begin{tikzpicture}[anchorbase,scale=.5]
      \node at (.5,.7) {\small $G_2$};
      \node (1) at (1,0) {$\bullet$};
      \node (2) at (2,0) {$\bullet$};
      \draw (1.center) -- (2.center)  node [midway, above] {\small $6$};
    \end{tikzpicture}
    \quad,\quad
    \begin{tikzpicture}[anchorbase,scale=.5]
      \node at (.5,.7) {\small $H_2$};
      \node (1) at (1,0) {$\bullet$};
      \node (2) at (2,0) {$\bullet$};
      \draw (1.center) -- (2.center)  node [midway, above] {\small $5$};
    \end{tikzpicture}
    \quad,\quad
    \begin{tikzpicture}[anchorbase,scale=.5]
      \node at (2.5,.7) {\small $H_3$};
      \node (1) at (1,0) {$\bullet$};
      \node (2) at (2,0) {$\bullet$};
      \node (3) at (3,0) {$\bullet$};
      \draw (1.center) -- (2.center)  node [midway, above] {\small $5$};
      \draw (2.center) -- (3.center);
    \end{tikzpicture}
    \quad,\quad
    \begin{tikzpicture}[anchorbase,scale=.5]
      \node at (2.5,.7) {\small $H_4$};
      \node (1) at (1,0) {$\bullet$};
      \node (2) at (2,0) {$\bullet$};
      \node (3) at (3,0) {$\bullet$};
      \node (4) at (4,0) {$\bullet$};
      \draw (1.center) -- (2.center)  node [midway, above] {\small $5$};
      \draw (2.center) -- (3.center);
      \draw (3.center) -- (4.center);
    \end{tikzpicture}    
    \quad,\quad
    \begin{tikzpicture}[anchorbase,scale=.5]
      \node at (.5,.7) {\small $I_n$};
      \node (1) at (1,0) {$\bullet$};
      \node (2) at (2,0) {$\bullet$};
      \draw (1.center) -- (2.center)  node [midway, above] {\small $n$};
    \end{tikzpicture}    
  \end{gather*}
\end{theorem}
A proof of this theorem is given for example in~\cite[Section 2.7]{Humphreys}.

\begin{definition}
  Let $\Gamma$ be a Coxeter diagram. Its associated \emph{Artin-Tits groups} $\B(\Gamma)$ is the group with :
  \begin{itemize}
\item generators $\sigma_i$ for $i \in V(\Gamma)$ the vertices of $\Gamma$;
\item relations:
  \[
\underbrace{\sigma_i\sigma_j\sigma_i\cdots}_\text{$m_{i,j}$ times}=\underbrace{\sigma_j\sigma_i\sigma_j\cdots}_\text{$m_{i,j}$ times}
\]
where $m_{i,j}$ is the label of the edge between the nodes corresponding to $\sigma_i$ and $\sigma_j$ ($m_{i,j}=2$ if the two vertices are not adjacent, and $m_{i,j}=\infty$ means that there is no relation).
\end{itemize}
\end{definition}

We will denote by $\B(\Gamma)^+$ the monoid generated by the $\sigma_i$'s with the same relations.

\begin{remark}
  By further imposing the generators to have order $2$, there is a natural surjection from the Artin-Tits group to the Coxeter group:
  \begin{align*}
    \B(\Gamma) &\twoheadrightarrow  G_\Gamma \\
    \sigma_i & \rightarrow s_i
    \end{align*}
\end{remark}

\begin{example}
  In type $A$, one gets the classical braid groups. Consider for example type $A_3$ labeled as follows:
  \[
    A_3=
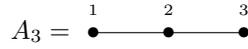
\begin{tikzpicture}[anchorbase]
      \node (1) at (1,0) {$\bullet$};
      \node at (1,.3) {\tiny $1$};
      \node at (1,-.3) {\tiny \vphantom{$1$}};
      \node (2) at (2,0) {$\bullet$};
      \node at (2,.3) {\tiny $2$};
      \node (3) at (3,0) {$\bullet$};
      \node at (3,.3) {\tiny $3$};
      \draw (1.center) -- (2.center);
      \draw (2.center) -- (3.center);
    \end{tikzpicture}
  \]
  The corresponding group $\B(\Gamma)$ has generators $\sigma_1$, $\sigma_2$ and $\sigma_3$. They admit diagrammatic depictions as follows:
  \[
    \sigma_1=\begin{tikzpicture}[anchorbase, scale=.5]
      \draw [thick] (1,0) to [out=90,in=-90] (0,1);
      \draw [white, line width=5] (0,0) to [out=90,in=-90] (1,1);
      \draw [thick] (0,0) to [out=90,in=-90] (1,1);
      \draw [thick] (2,0) -- (2,1);
      \draw [thick] (3,0) -- (3,1);
    \end{tikzpicture}
    \quad,\quad
        \sigma_2=\begin{tikzpicture}[anchorbase, scale=.5]
      \draw [thick] (0,0) -- (0,1);
      \draw [thick] (2,0) to [out=90,in=-90] (1,1);
      \draw [white, line width=5] (1,0) to [out=90,in=-90] (2,1);
      \draw [thick] (1,0) to [out=90,in=-90] (2,1);
      \draw [thick] (3,0) -- (3,1);
      \end{tikzpicture}
    \quad,\quad
    \sigma_3=\begin{tikzpicture}[anchorbase, scale=.5]
      \draw [thick] (0,0) -- (0,1);
      \draw [thick] (1,0) -- (1,1);
      \draw [thick] (3,0) to [out=90,in=-90] (2,1);
      \draw [white, line width=5] (2,0) to [out=90,in=-90] (3,1);
      \draw [thick] (2,0) to [out=90,in=-90] (3,1);
      \end{tikzpicture}
    \]
    Nodes $1$ and $3$ not sharing an edge governs the fact that $\sigma_1$ and $\sigma_3$ commute, while nodes $1$ and $2$ being adjacent means that the corresponding generators braid: $\sigma_1\sigma_2\sigma_2=\sigma_2\sigma_1\sigma_2$ (and similarly for nodes $2$ and $3$).
    \[
      \begin{tikzpicture}[scale=.5, anchorbase]
        \draw [thick] (1,0) to [out=90,in=-90] (0,1);
        \draw [white, line width=5] (0,0) to [out=90,in=-90] (1,1);
        \draw [thick] (0,0) to [out=90,in=-90] (1,1);
      \draw [thick] (2,0) -- (2,1);
      \draw [thick] (3,0) -- (3,1);
      \draw [thick] (0,1) -- (0,2);
      \draw [thick] (1,1) -- (1,2);
      \draw [thick] (3,1) to [out=90,in=-90] (2,2);
      \draw [white, line width=5] (2,1) to [out=90,in=-90] (3,2);
      \draw [thick] (2,1) to [out=90,in=-90] (3,2);
    \end{tikzpicture}
    =
          \begin{tikzpicture}[scale=.5, anchorbase]
        \draw [thick] (1,1) to [out=90,in=-90] (0,2);
        \draw [white, line width=5] (0,1) to [out=90,in=-90] (1,2);
        \draw [thick] (0,1) to [out=90,in=-90] (1,2);
      \draw [thick] (2,1) -- (2,2);
      \draw [thick] (3,1) -- (3,2);
          \draw [thick] (0,0) -- (0,1);
      \draw [thick] (1,0) -- (1,1);
      \draw [thick] (3,0) to [out=90,in=-90] (2,1);
      \draw [white, line width=5] (2,0) to [out=90,in=-90] (3,1);
      \draw [thick] (2,0) to [out=90,in=-90] (3,1);
    \end{tikzpicture}
    \quad,\quad
    \begin{tikzpicture}[scale=.5, anchorbase]
        \draw [thick] (1,0) to [out=90,in=-90] (0,1);
        \draw [white, line width=5] (0,0) to [out=90,in=-90] (1,1);
        \draw [thick] (0,0) to [out=90,in=-90] (1,1);
      \draw [thick] (2,0) -- (2,1);
      \draw [thick] (3,0) -- (3,1);
      \draw [thick] (0,1) -- (0,2);
      \draw [thick] (2,1) to [out=90,in=-90] (1,2);
      \draw [white, line width=5] (1,1) to [out=90,in=-90] (2,2);
      \draw [thick] (1,1) to [out=90,in=-90] (2,2);
      \draw [thick] (3,1) -- (3,2);
      \draw [thick] (1,2) to [out=90,in=-90] (0,3);
        \draw [white, line width=5] (0,2) to [out=90,in=-90] (1,3);
        \draw [thick] (0,2) to [out=90,in=-90] (1,3);
      \draw [thick] (2,2) -- (2,3);
      \draw [thick] (3,2) -- (3,3);
    \end{tikzpicture}
    =
          \begin{tikzpicture}[scale=.5, anchorbase]
      \draw [thick] (0,0) -- (0,1);
      \draw [thick] (2,0) to [out=90,in=-90] (1,1);
      \draw [white, line width=5] (1,0) to [out=90,in=-90] (2,1);
      \draw [thick] (1,0) to [out=90,in=-90] (2,1);
      \draw [thick] (3,0) -- (3,1);
      \draw [thick] (1,1) to [out=90,in=-90] (0,2);
      \draw [white, line width=5] (0,1) to [out=90,in=-90] (1,2);
      \draw [thick] (0,1) to [out=90,in=-90] (1,2);
      \draw [thick] (2,1) -- (2,2);
      \draw [thick] (3,1) -- (3,2);
      \draw [thick] (0,2) -- (0,3);
      \draw [thick] (2,2) to [out=90,in=-90] (1,3);
      \draw [white, line width=5] (1,2) to [out=90,in=-90] (2,3);
      \draw [thick] (1,2) to [out=90,in=-90] (2,3);
      \draw [thick] (3,2) -- (3,3);
    \end{tikzpicture}
    \]
For the equivalence between group-theoretic and geometric descriptions (and for much more!), see~\cite[Chapter1]{Birman_book}.  
\end{example}

There are only few examples where Artin-Tits groups can be presented by pictures. In addition to type $A$, one can cite:
\begin{itemize}
\item type $\tilde{A}$ (affine type $A$) and $B/C$, with braids in an annulus:
  
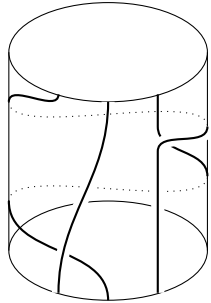
\[
  \begin{tikzpicture}[anchorbase, scale=.66]
    \draw [opacity=.4] (0,0) ellipse (2 and 1);
    \draw [opacity=.4] (0,4) ellipse (2 and 1);
    \draw [opacity=.4] (-2,4) -- (-2,0);
    \draw [opacity=.4] (2,4) -- (2,0);
    \draw [thick] (0,-1) to [out=90,in=-90] (-2,1);
    \draw [dotted] (-2,1) to [out=90,in=-90,looseness=.5] (2,1.5);
    \draw [thick] (2,1.5) to [out=90,in=-90] (1,2.5) -- (1,3.15);
    \draw [white, line width=5] (-1,-.85) to [out=90,in=-90] (0,3);
    \draw [thick] (-1,-.85) to [out=90,in=-90] (0,3);
    \draw [white, line width=5] (2,2.5) to [out=-90,in=90] (1,2) -- (1,-.85);
    \draw [thick] (-1,3.15) to [out=-90,in=90] (-2,3);
    \draw [dotted] (-2,3) to [out=-90,in=90,looseness=.5] (2,2.5);
    \draw [thick] (2,2.5) to [out=-90,in=90] (1,2) -- (1,-.85);
    \end{tikzpicture}
  \]
\item type $D$ and affine type $D$, where there is an orbifold relation: see~\cite{Allcock}.
\end{itemize}

\begin{example}
A class of Artin-Tits groups that seems very dear to geometric group theorists is the one of right-angled Artin groups. There any two nodes in the diagram are either unrelated or related by an edge labeled $\infty$. So any two generators either commute or generate a free subgroup. We refer to~\cite{Koberda_RAAG} for a recent survey on this topic.
\end{example}

The projection $\B(\Gamma)\twoheadrightarrow G_\Gamma$ admits a set-theoretic section as follows:
\begin{align}
  G_\Gamma & \rightarrow \B(\Gamma) \\
  s_{i_1}\cdots s_{i_k} \;\text{reduced word} \rightarrow \sigma_{i_1}\rightarrow \sigma_{i_k}
\end{align}
It can be checked that the definition doesn't depend on a choice of a reduced word. This is called the positive lift from the Coxeter group to the braid group.

In finite type, the long element $\delta$ lifts to a special element that we will denote $\Delta$.

\begin{example}
  In type $A_2$,
  
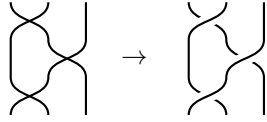
\[
  \begin{tikzpicture}[anchorbase, scale=.5]
    \draw [thick] (0,0) to [out=90,in=-90] (1,1) to [out=90,in=-90] (2,2) -- (2,3);
    \draw [thick] (1,0) to [out=90,in=-90] (0,1) -- (0,2) to [out=90,in=-90] (1,3);
    \draw [thick] (2,0) -- (2,1) to [out=90,in=-90] (1,2) to [out=90,in=-90] (0,3);
  \end{tikzpicture}
  \quad \rightarrow \quad
  \begin{tikzpicture}[anchorbase, scale=.5]
    \draw [thick] (2,0) -- (2,1) to [out=90,in=-90] (1,2) to [out=90,in=-90] (0,3);
    \draw [white, line width=5] (1,0) to [out=90,in=-90] (0,1) -- (0,2) to [out=90,in=-90] (1,3);
    \draw [thick] (1,0) to [out=90,in=-90] (0,1) -- (0,2) to [out=90,in=-90] (1,3);
    \draw [white, line width=5] (0,0) to [out=90,in=-90] (1,1) to [out=90,in=-90] (2,2) -- (2,3);
    \draw [thick] (0,0) to [out=90,in=-90] (1,1) to [out=90,in=-90] (2,2) -- (2,3);
  \end{tikzpicture}
  \]

  In type $A$ $\Delta$ is the positive half-twist.
\end{example}

In finite type, it was shown in~\cite{Garside,ElrifaiMorton} that the interval $[1,\Delta]$ defined as follows is a lattice:
\[
[1,\Delta]:=\{ \beta \in \B(\Gamma)^+ | \exists \beta' \in \B(\Gamma^+) \;\text{such that}\; \beta \beta'=\Delta\}
\]
This means that we can take gcd's and lcm's. This is the beginning of Garside theory, which in particular yields normal forms for braids. We won't exploit this structure a lot here, but another one that is closely related in spirit, and often referred to as the dual Garside structure -- the one defined above being the classical one.

\subsection{Dual generators, dual Garside structure} \label{sec:dualgen}

Choose a Coxeter element $c\in W$, which is a product of the generators $s_i$ in some order. It can be lifted into a positive word of the braid group, giving an element $\gamma$. Our favorite example will be the following one, in type $A$:
\[
c=s_1s_2\cdots s_n,\quad \gamma=\sigma_1\sigma_2\cdots \sigma_n.
\]
Such a choice corresponds to an orientation of $\Gamma$, with the convention that an edge maps from $i$ to $j$ if $s_i$ appears before $s_j$ in the expression of $c$.

In this version of the theory, we'll try to lift reflections $t_\alpha\in W$ (those are indexed by roots, but we won't really use the root formalism much). In the Weyl group, reflections can be obtained as follows: $t=ws_iw^{-1}$. Trying to mimic this, let us now denote $[1,\gamma]$ the set:
\[
[1,\gamma]=\{\text{left subwords of writings }\gamma=T_1\cdots T_n,\;T_i=\beta_i\sigma_j\beta_i^{-1}\}
\]
To highlight the choices made, consider the case of $s_1s_2s_1=t_{\alpha_1+\alpha_2}$ in type $A_3$. When trying to lift this to the braid group, one has several equally natural options: $\sigma_1\sigma_2\sigma_1$, using the positive lift? $\sigma_1\sigma_2\sigma_1^{-1}$, which retains some symmetry since $\sigma_1\sigma_2\sigma_1^{-1}=\sigma_2^{-1}\sigma_1\sigma_2$? Or why not $\sigma_1\sigma_2^{-1}\sigma_1$? Here the magic is that the choice of $\gamma$ will give us a unique lift for each reflection, and that those will arrange nicely.

This set of generator was studied by Birman-Ko-Lee~\cite{BKL} and Bessis~\cite{Bessis}, who proved the following theorem:
\begin{theorem}
  $[1,\gamma]$ is a lattice.
\end{theorem}

We will be very much interested in generators $T_\alpha$ in $[1,\gamma]$.

\begin{example}
  In type $A$ with the choice $\gamma=\sigma_1\cdots \sigma_n$, each root $\alpha=\alpha_i+\dots + \alpha_{i+k}$ gets assigned a unique generator:
  \[T_\alpha=\sigma_i\cdots \sigma_{i+k-1}\sigma_{i+k}\sigma_{i+k-1}^{-1}\cdots \sigma_{i}.\]
\end{example}

\begin{example}
  Here is the lattice in type $A_3$ with choice of $\gamma=\sigma_1\sigma_2\sigma_3$.

  
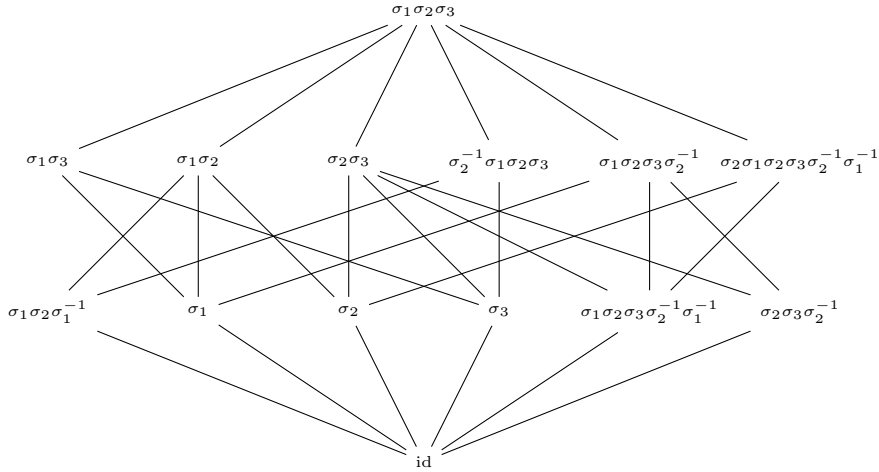
\[
  \begin{tikzpicture}[anchorbase]
    \node (1) at (0,0) {\tiny $\id$};
    \node (13) at (-5,2) {\tiny $\sigma_1\sigma_2\sigma_1^{-1}$};
    \node (12) at (-3,2) {\tiny $\sigma_1$};
    \node (23) at (-1,2) {\tiny $\sigma_2$};
    \node (34) at (1,2) {\tiny $\sigma_3$};
    \node (14) at (3,2) {\tiny $\sigma_1\sigma_2\sigma_3\sigma_2^{-1}\sigma_1^{-1}$};
    \node (24) at (5,2) {\tiny $\sigma_2\sigma_3\sigma_2^{-1}$};
    \node (12x34) at (-5,4) {\tiny $\sigma_1\sigma_3$};
    \node (123) at (-3,4) {\tiny $\sigma_1\sigma_2$};
    \node (234) at (-1,4) {\tiny $\sigma_2\sigma_3$};
    \node (134) at (1,4) {\tiny $\sigma_2^{-1}\sigma_1\sigma_2\sigma_3$};
    \node (124) at (3,4) {\tiny $\sigma_1\sigma_2\sigma_3\sigma_2^{-1}$};
    \node (23x14) at (5,4) {\tiny $\sigma_2 \sigma_1\sigma_2\sigma_3\sigma_2^{-1}\sigma_1^{-1}$};
    \node (1234) at (0,6) {\tiny $\sigma_1\sigma_2\sigma_3$};
    \draw (1) -- (13);
    \draw (1) -- (12);
    \draw (1) -- (23);
    \draw (1) -- (34);
    \draw (1) -- (14);
    \draw (1) -- (24);
    \draw (12) -- (123);
    \draw (12) -- (12x34);
    \draw (12) -- (124);
    \draw (23) -- (123);
    \draw (23) -- (23x14);
    \draw (23) -- (234);
    \draw (34) -- (134);
    \draw (34) -- (12x34);
    \draw (34) -- (234);
    \draw (14) -- (124);
    \draw (14) -- (23x14);
    \draw (14) -- (234);
    \draw (13) -- (123);
    \draw (13) -- (134);
    \draw (24) -- (124);
    \draw (24) -- (234);
    \draw (123) -- (1234);
    \draw (234) -- (1234);
    \draw (134) -- (1234);
    \draw (124) -- (1234);
    \draw (12x34) -- (1234);
    \draw (23x14) -- (1234);
  \end{tikzpicture}
  \]
\end{example}

The elements in the interval $[1,\gamma]$ play a central role in the Garside structure for the whole Artin-Tits group. In particular, they form the letters of the alphabet in which one expresses the Garside normal form. We'll see those show up again later.

\subsection{Burau representation}

The Coxeter group $G_\Gamma$ has an action on a finite-dimensional representation, from which one can study the group properties: see~\cite{BjornerBrenti} for example. The underlying complex vector space is:
\[
V_1=\langle \alpha_1,\cdots, \alpha_n\; \text{for}\; i\in V(\Gamma) \rangle_\C
\]
equipped with a non-degenerate pairing as follows:
\[
\langle \alpha_i,\alpha_j\rangle =
\begin{cases}
  2 \; \text{if} \; i=j; \\
  0 \; \text{if} \; m_{ij}=2\;\text{($i$ and $j$ not adjacent)} \\
  -2\cos(\frac{\Pi}{m_{ij}})\; \text{if}\;m_{ij}\neq \infty \\
  -2\; \text{if} \; m_{ij}=\infty
\end{cases}
\]
Note that the $m_{ij}=2$ case is compatible with the definition for general $m_{ij}$.

Now $G_\Gamma$ acts on $V_1$ by reflections:
\[
s_i(\alpha_j)=\alpha_j-\langle \alpha_i,\alpha_j\rangle \alpha_i.
\]

Again, a key point is that this representation is faithful: see for example \cite[Theorem 4.2.7]{BjornerBrenti}.

Of course, since the projection from $B(\Gamma)$ to $G_\Gamma$ has a kernel, the induced action of the Artin-Tits groups on the same space won't be faithful. However, it is a natural strategy to try to q-deform it, especially from the perspective of Hecke algebras. The result will be the Burau representation, that we'll describe now.

Let us $q$-deform $V_1$ into $V_q$ a $\C(q)$-vector space generated by the roots $\alpha_1,\dots, \alpha_n$. Then the pairing also $q$-deforms:
\[
\langle \alpha_i,\alpha_j\rangle =\begin{cases}
1+q^2 \;\text{if}\;i=j \\
-2q\cos(\frac{\Pi}{m_{ij}})\; \text{if}\;i\neq j\;\text{and}\; m_{ij}<\infty \\
-2\;\text{if}\;m_{ij}=\infty
\end{cases}
\]

\begin{definition}
  The \emph{Burau representation} is an action of $\B(\Gamma)$ on $V_q$ defined on generators $\sigma_i\in \B(\Gamma)$ by:
  \[
    \sigma_i(\alpha_j)=\alpha_j-\langle \alpha_i,\alpha_j\rangle \alpha_i.
  \]
\end{definition}

\begin{remark}
  One can check that this action factors through the Hecke algebra, which is a quadratic quotient of the group ring of the braid group. In type $A$, this is also related to the following appearance of the Burau representation. Consider $U=\C(q)^2$ the 2-dimensional irreducible representation of $U_q(\mathfrak{sl}_2)$. Then Schur-Weyl duality asserts that there is an action of the braid group $\B(A_{n-1})$ on $U^{\otimes n}$, dual to the one of $U_q(\mathfrak{sl}_2)$. Decomposing this $n$-fold tensor into $\mathfrak{sl}_2$ weight spaces, one gets:
  \[
U^{\otimes n}=\bigoplus_{k=-n .. n\;\text{by} \; 2} W_n
\]
The $\B(A_{n-1})$ action respects the direct sum decomposition, and one can check that $W_{n-2}$ is the Burau representation.
\end{remark}

\begin{example}
  In type $A$, one can check that:
  \[
    \sigma_i(\alpha_j)=\begin{cases} -q^2\alpha_i \;\text{if}\; i=j, \\
      \alpha_i-q\alpha_j \;\text{if}\; |i-j|=1, \\
      \alpha_j\;\text{otherwise}.
      \end{cases}
  \]
\end{example}

As for faithfulness, the situation is not very good though.
\begin{theorem}[~\cite{MagnusPeluso, Moody, LongPaton, Bigelow_Burau}] \label{thm:BurauKernel}
  The Burau representation is faithful in type $A_2$, and not faithful in any group containing $\B(A_4)$ as a subgroup.
\end{theorem}

We will delay the (sketched) proof until Section~\ref{sec:sphTwAgain}. Faithfulness in type $A_2$ is classical (see the discussion in Birman's book~\cite{Birman_book}), but the most complicated proof ever can be found in Theorem~\ref{thm:3BurauFaithful}.

\begin{problem}
  What about $A_3$? $D_4$? $\tilde{A}_3$? $\tilde{A}_2$?
\end{problem}

In type $A_3$, the problem has been reduced by Birman~\cite{Birman_book} to showing that it is faithful on the free subgroup generated by the following two braids:
\[
\sigma_1\sigma_3^{-1} \quad \text{and} \quad (\sigma_1\sigma_2\sigma_1^{-1})(\sigma_3\sigma_2\sigma_3^{-1})
\]
In other words, do the following two matrices generate a free group?
\[
\begin{pmatrix}
  -q^2 & -q & 0 \\
  0 & 1 & 0 \\
  0 & -q^{-1} & -q^{-2}
\end{pmatrix}
\quad \text{and} \quad
\begin{pmatrix}
  1-q^2 & q^{-1} & 1 \\
  -q^{-1}+q^3 & -q^{-2} & 0 \\
  q^2 & q^{-3} & 0
\end{pmatrix}
\]

The Burau representation was originally defined by Burau~\cite{Burau} in the 30's. It has some ties with the Alexander polynomial of knots. The presentation I gave makes quite clear the fact that it is unitary (in a $q$-version), but this result was originally proved by Squier~\cite{Squier}.

The lack of faithfulness of the Burau representation was fixed in the early 2000's in two different and apparently unrelated manners. Khovanov and Seidel's categorical lift is one of them, and the main topic of these notes. The other version is a genuine finite-dimensional representation, the so-called LKB representation (LKB stands for Lawrence-Krammer-Bigelow). The definition in type $A$ comes from work of Lawrence in the early 90's~\cite{Lawrence}, and it has been shown to be faithful both by Bigelow~\cite{Bigelow_linearity} (using some intersection count in a cover of a configuration space of points) and Krammer~\cite{Krammer} (using Garside theory). This is the original proof that braid groups, and Artin-Tits groups in finite type, are linear.

In general, this leaves open the following question.
\begin{problem}
  Are Artin-Tits groups linear?
\end{problem}

Outside of finite type, Paris proved:
\begin{theorem}[\cite{Paris_monoid}]
  The LKB representation is a faithful representation of the positive monoid.
\end{theorem}

\begin{corollary}
  The monoid embeds in the braid group.
\end{corollary}
\begin{proof}
  We have: $\rho: \B^+(\Gamma)\rightarrow \End(H)$ faithful as a monoid map. It induces a map from the group of fractions, and one has the following commutative diagram:
  \[
    \begin{tikzpicture}[anchorbase]
      \node (TL) at (0,0) {$\B^+(\Gamma)$};
      \node (TR) at (2,0) {$\B(\Gamma)$};
      \node (BR) at (2,-1.5) {$\End(H)$};
      \draw [->] (TL) -- (TR);
      \draw [->] (TR) -- (BR);
      \draw [right hook->] (TL) -- (BR);
    \end{tikzpicture}
  \]
  This implies that the horizontal arrow needs to be injective.
\end{proof}

In type $A$, one can show that the LKB representation is a $t$-deformation of the second symmetric power of the Burau representation. This is very special to type $A$, where the number of positive roots is $\frac{n(n-1)}{2}$ if there are $n$ simple roots. For more material on the LKB representations from a representation theoretic perspective, we refer to~\cite{Martel,LTV_LKB} and references therein.

\begin{problem}
  What is the link between Burau and LKB in general?
\end{problem}

For considerations from geometric group theory as well as categorification, it would also be quite valuable to know that the LKB representations are unitary -- and for which quadratic form. To the best of my knowledge, this is an open question.
\begin{problem}
  Is the LKB representation unitary in general?
\end{problem}

\section{Khovanov-Seidel 2-representation}

We now aim at categorifying the Burau representation, following Khovanov and Seidel's work~\cite{KhS}. As mentioned earlier, they proved that their 2-representation is faithful around the same time that Bigelow and Krammer proved the faithfulness of the LKB representation: braid groups were proven to be linear and 2-linear at the same time.

Recall that the Burau representations was built from the $\C(q)$ vector space $V_q(\Gamma)$, with a pairing and generators of the Artin-Tits groups acting by reflections. For simplicity, we'll restrict to $\Gamma$ simply-laced. More general definitions can be found in~\cite{Licata_free} for the free group and~\cite{HengNge} in type $B$.

What we are looking for is a category $\calC_\Gamma$ with Grothendieck group isomorphic to $V_q(\Gamma)$, and an assignment of an autoequivalence $KhS(g)\in \Aut(\calC)$ for each element $g\in G$.

Here is a rough sketch of what we will build:
\begin{itemize}
\item $\calC_\Gamma=\Kom_h(A_\Gamma-pmod)$ for some algebra $A_\Gamma$;
\item $\alpha_i \;\leftrightarrow \; P_i$ indecomposable projective;
\item $\sigma_i \; \leftrightarrow \; \Sigma_i$ complex of bimodules;
\item $\langle -,-\rangle \; \leftrightarrow \qdim(\HOM(-,-))$.
\end{itemize}

Note that what we are going to define is a weak action: no claim is made about the homotopies $\Sigma_1\Sigma_2\Sigma_1\simeq \Sigma_2\Sigma_1\Sigma_2$. 

It should also be noted that $\calC_\Gamma$ is a $2$-Calabi-Yau category, which gives a conceptual framework to several of the objects I'm going to define later. Since I know very little about this conceptual framework, I'll just ignore it.

Before we start, let's make a last comment about the pairing on the Burau space, that's preserved by this action. This pairing preservation is a feature expected if there exists a categorification by projectives. Assume that $\calC$ is a category of projective modules acted on by a group $G$. The Grothendieck group $K_0(\calC)$ will have a pairing (or $q$-pairing if one uses graded dimension):
\[
  \langle [A],[B] \rangle= \dim(\Hom_\calC(A,B))
\]
and this pairing will be naturally preserved by the action, since elements $g\in G$ act by endofunctors of $\calC$:
\[\Hom(gA,gB)\simeq \Hom(A,B)
\]

\subsection{Zig-zag algebra and projective modules}

Recall that we have restricted ourselves to the case of $\Gamma$ a simply laced diagram with labels $2$ or $3$ only.

Let us restart from the Dynkin diagram $\Gamma$. One can associate to it its oriented double quiver $\vec{\Gamma}$ by doubling all edges and giving each copy a different orientation.

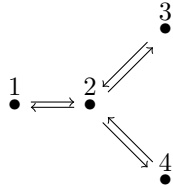
\[
  \begin{tikzpicture}[anchorbase]
    \node (1)  at (1,0) {$\bullet$};
    \node (2) at (2,0) {$\bullet$};
    \node (3) at (3,1) {$\bullet$};
    \node (4) at (3,-1) {$\bullet$};
    \node [above] at (1) {$1$};
    \node [above] at (2) {$2$};
    \node [above] at (3) {$3$};
    \node [above] at (4) {$4$};
    \begin{scope}[transform canvas={yshift=1}]
      \draw [->] (1) -- (2);
    \end{scope}
    \begin{scope}[transform canvas={yshift=-1}]
      \draw [<-] (1) -- (2);
    \end{scope}
    \begin{scope}[transform canvas={yshift=-1,xshift=1}]
      \draw [->] (2) -- (3);
    \end{scope}
    \begin{scope}[transform canvas={yshift=1,xshift=-1}]
      \draw [<-] (2) -- (3);
    \end{scope}
    \begin{scope}[transform canvas={yshift=-1,xshift=-1}]
      \draw [->] (2) -- (4);
    \end{scope};
    \begin{scope}[transform canvas={yshift=1,xshift=1}]
      \draw [<-] (2) -- (4);
    \end{scope};
  \end{tikzpicture}
\]

We then define the zig-zag algebra:
\[
  A_{\Gamma}=
  \frac{
    \mathrm{Path}(\vec{\Gamma})
    }{\begin{tikzpicture}[anchorbase, scale=.5] \node (1) at (1,0) {$\bullet$}; \node (2) at (2,0) {$\bullet$}; \node (3) at (3,0) {$\bullet$}; \draw [opacity=.5] (1.center) -- (2.center) -- (3.center); \draw [->] (1.north) to [out=30,in=150] (2.north) to [out=30,in=150] (3.north); \end{tikzpicture} =0\;,\quad \begin{tikzpicture}[anchorbase] \node (1) at (1,0) {$\bullet$}; \node (2) at (2,0) {$\bullet$}; \node (3) at (3,0) {$\bullet$}; \draw [opacity=.5] (1.center) -- (2.center) -- (3.center); \draw [->] (2.center) to [out=120,in=60] (1.east) to [out=-60,in=-120] (2.west); \end{tikzpicture}=\begin{tikzpicture}[anchorbase] \node (1) at (1,0) {$\bullet$}; \node (2) at (2,0) {$\bullet$}; \node (3) at (3,0) {$\bullet$}; \draw [opacity=.5] (1.center) -- (2.center) -- (3.center); \draw [->] (2.center) to [out=60,in=120] (3.west) to [out=-120,in=-60] (2.east); \end{tikzpicture}}
\]

The path algebra, as a $\C$-module, is generated by all paths in the quivers. Paths are simply sequences of arrows so that the target of each arrow equals the source of the next one.

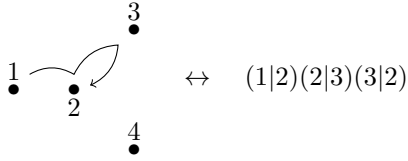
\begin{example}
  \[
  \begin{tikzpicture}[anchorbase,scale=.8]
    \node (1)  at (1,0) {$\bullet$};
    \node (2) at (2,0) {$\bullet$};
    \node (3) at (3,1) {$\bullet$};
    \node (4) at (3,-1) {$\bullet$};
    \node [above] at (1) {$1$};
    \node [below] at (2) {$2$};
    \node [above] at (3) {$3$};
    \node [above] at (4) {$4$};
    \draw [->] (1) to [bend left] (2.north) to [bend left] (3.south west) to [bend left] (2);
  \end{tikzpicture}
  \quad \leftrightarrow \quad (1|2)(2|3)(3|2)
  \]
\end{example}

Multiplication (composition) is given by concatenation of path. When concatenation is not possible, then the composition yields zero.

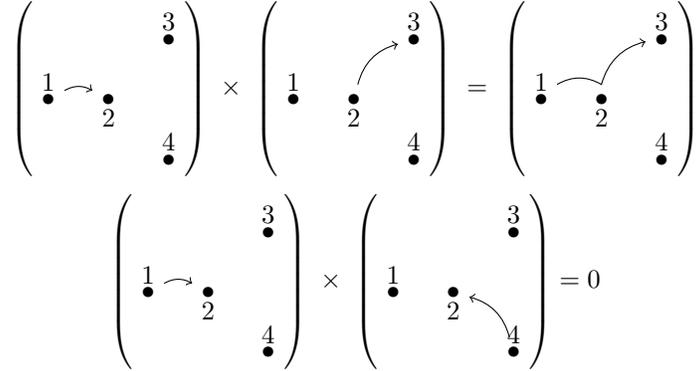
\begin{example}
\[
\left(  \begin{tikzpicture}[anchorbase,scale=.8]
    \node (1)  at (1,0) {$\bullet$};
    \node (2) at (2,0) {$\bullet$};
    \node (3) at (3,1) {$\bullet$};
    \node (4) at (3,-1) {$\bullet$};
    \node [above] at (1) {$1$};
    \node [below] at (2) {$2$};
    \node [above] at (3) {$3$};
    \node [above] at (4) {$4$};
    \draw [->] (1) to [bend left] (2);
\end{tikzpicture}
\right)
\;\times \;
\left(
  \begin{tikzpicture}[anchorbase,scale=.8]
    \node (1)  at (1,0) {$\bullet$};
    \node (2) at (2,0) {$\bullet$};
    \node (3) at (3,1) {$\bullet$};
    \node (4) at (3,-1) {$\bullet$};
    \node [above] at (1) {$1$};
    \node [below] at (2) {$2$};
    \node [above] at (3) {$3$};
    \node [above] at (4) {$4$};
    \draw [->] (2) to [bend left] (3);
  \end{tikzpicture}
  \right)
  \;
  =
  \;
  \left(
  \begin{tikzpicture}[anchorbase,scale=.8]
    \node (1)  at (1,0) {$\bullet$};
    \node (2) at (2,0) {$\bullet$};
    \node (3) at (3,1) {$\bullet$};
    \node (4) at (3,-1) {$\bullet$};
    \node [above] at (1) {$1$};
    \node [below] at (2) {$2$};
    \node [above] at (3) {$3$};
    \node [above] at (4) {$4$};
    \draw [->] (1) to [bend left] (2.north) to [bend left] (3);
  \end{tikzpicture}
  \right)
  \]
  
\[
\left(  \begin{tikzpicture}[anchorbase,scale=.8]
    \node (1)  at (1,0) {$\bullet$};
    \node (2) at (2,0) {$\bullet$};
    \node (3) at (3,1) {$\bullet$};
    \node (4) at (3,-1) {$\bullet$};
    \node [above] at (1) {$1$};
    \node [below] at (2) {$2$};
    \node [above] at (3) {$3$};
    \node [above] at (4) {$4$};
    \draw [->] (1) to [bend left] (2);
\end{tikzpicture}
\right)
\;\times \;
\left(
  \begin{tikzpicture}[anchorbase,scale=.8]
    \node (1)  at (1,0) {$\bullet$};
    \node (2) at (2,0) {$\bullet$};
    \node (3) at (3,1) {$\bullet$};
    \node (4) at (3,-1) {$\bullet$};
    \node [above] at (1) {$1$};
    \node [below] at (2) {$2$};
    \node [above] at (3) {$3$};
    \node [above] at (4) {$4$};
    \draw [->] (4) to [bend right] (2);
  \end{tikzpicture}
  \right)
 =
0 \]  

\end{example}

Primitive idempotents consist simply in staying at vertex $(i)$, and are denoted $e_i$. One gets:
\[
1=\sum_i e_i
\]

It is easily checked that there are no paths of length more than (or equal to) $3$.

Originally, Khovanov and Seidel based their algebra in the following sense: one of the extremal nodes of the Dynkin diagram has as extra relation that its loop is zero.

By adding a node, any of the zig-zag algebras defined here embeds in an algebra based at a node. Two reasons arise to like those better: they contain exceptional objects, which might be useful to create a bridge with algebraic geometry, and they are closer to the quadratic duals of the preprojective algebra of the corresponding quiver: see Appendix~\ref{appendix:quadratic_dual} for details.

A natural grading on this algebra is given by the path length: all relations preserve the grading. Degree shifts for the path length degree will be denoted by angle brackets $\langle - \rangle$.

Other gradings can be defined, from additional data. An example we will care about comes from the choice of an orientation of the diagram $\Gamma$. One can decide that a generating path has degree zero if it has the favorite orientation, and degree one if it goes against the chosen orientation. Again, it can be checked that all relations are homogeneous. We'll refer to this one as the orientation grading, and use curly brackets $\{-\}$ for shifts in this degree.

\begin{example}
  In type $A$, our main example will be for the following choice of orientation:
  \[
  \begin{tikzpicture}[anchorbase]
    \node (1) at (1,0) {$\bullet$};
    \node [above] at (1) {\small $1$};
    \node (2) at (2,0) {$\bullet$};
    \node [above] at (2) {\small $2$};
    \node (3) at (3,0) {$\bullet$};
    \node [above] at (3) {\small $3$};
    \node (4) at (4,0) {$\bullet$};
    \node [above] at (4) {\small $4$};
    \node (dots) at (4.5,0) {$\dots$};
    \draw [->] (1) -- (2);
    \draw [->] (2) -- (3);
    \draw [->] (3) -- (4);
  \end{tikzpicture}
  \]

  Then:
  \[
  \deg\left(
  \begin{tikzpicture}[anchorbase]
    \node (i) at (1,0) {$\bullet$};
    \node [below] at (i) {\small $i$};
    \node (i+1) at (2,0) {$\bullet$};
    \node [below] at (i+1) {\small $i+1$};
    \draw [->] (i) to [bend left] (i+1);
  \end{tikzpicture}
  \right)
  =0
  ,\quad
  \deg\left(
  \begin{tikzpicture}[anchorbase]
    \node (i) at (1,0) {$\bullet$};
    \node [below] at (i) {\small $i$};
    \node (i+1) at (2,0) {$\bullet$};
    \node [below] at (i+1) {\small $i+1$};
    \draw [->] (i+1) to [bend right] (i);
  \end{tikzpicture}
  \right)
  =1 
  \]
\end{example}

\subsection{Projectives}

Recall that we have a complete family of primitive idempotents:
\[1=\sum_{i\in V(\Gamma)} e_i\]
Thus, one gets:
\[
A_\Gamma=\bigoplus_i A_\Gamma e_i=\bigoplus_i e_i A_\Gamma
\]
Let us denote:
\[
P_i=A_{\Gamma}e_i,\quad Q_i=e_i A_\Gamma.
\]
The $P_i$'s are the graded indecomposable left projective modules over $A_\Gamma$, and the $Q_i$'s are the graded indecomposable right projective modules. In terms of paths, $P_i$ consists in those paths that end at vertex $i$, while $Q_i$ is made of path starting at vertex $i$.

\begin{example}
  In type $A_3$ (or greater type $A$),
  \[
  P_2=\langle e_2, (3|2), (1|2), x_2:=(2|1|2)=(2|3|2)\rangle
  \]
\end{example}

Let us finally define the main player of these lecture notes:
\[
\calC_\Gamma=\Kom_h(A_\Gamma-lpmod)
\]
In words, $\calC_\Gamma$ is the homotopy category of left projective $A_\Gamma$ modules.

$\Hom$-spaces between $P_i$'s are very easy to understand:
\begin{lemma}
  \[
  \HOM_{\calC_{\Gamma}}(P_i,P_j)=
  \begin{cases}
    \C\oplus \C\langle 2\rangle\{1\} \; \text{if}\; i=j \\
    \C\langle 1\rangle\{1\} \; \text{if}\; i\leftarrow j \\
    \C\langle 1\rangle \; \text{if}\; i\rightarrow j \\
    \{0\} \;\text{otherwise}.
  \end{cases}
  \]
\end{lemma}

\begin{proof}
  Let's consider the case when $i=j$. $P_i$ being a free module, a morphism $\phi$ is determined by the image of $e_i$. Since $e_ie_i=e_i$, $e_i\phi(e_i)=e_i$, which means that:
  \[\phi(e_i)=\sum a_j e_iy_j e_i.
  \]
  The $y_j$'s are path from the node $i$ to itself, that is, linear combinations of the idempotent $e_i$ and the loop at $i$. The first one is of degree zero for both gradings and the second one of degree two in path length and one in orientation grading.

  The other two cases follow from similar arguments.
\end{proof}

The following lemma is closely related to the previous one, and will be useful when we consider tensor product of bimodules.

\begin{lemma}
  \[
  Q_j\otimes P_i\simeq
  \begin{cases}
    \C\oplus \C\langle 2\rangle\{1\} \;\text{if}\; i=j \\
    \C\langle 1\rangle\{1\} \; \text{if}\; i \leftarrow j\\
    \C\langle 1\rangle \; \text{if}\; i \rightarrow j\\
    \{0\} \;\text{otherwise}.
  \end{cases}
  \]
\end{lemma}
\begin{proof}
  This is a simple path count.
\end{proof}

We can also phrase this all in terms of adjunctions.

\begin{lemma}
  The functor given by tensoring on the right with $Q_i\langle -2\rangle\{-1\}$ is right adjoint to tensoring with $P_i$, and tensoring with $Q_i$ is left adjoint to it.
\end{lemma}

\subsection{Complexes}

Let us first notice that we have reached the first goal of lifting up the space $V_q$ together with its pairing.
\begin{lemma}
  \[
K_0(\calC_{\Gamma}) \simeq V_q(\Gamma)
\]
with:
\[
\qdim(\HOM(P_i,P_j))=\langle \alpha_i,\alpha_j\rangle_q.
\]
  \end{lemma}

Complexes of projective modules will typically look like the following:
\[
\begin{tikzpicture}[anchorbase]
  \node (0) at (0,0) {$P_1$};
  \node (1) at (2,0) {$P_2\langle -1\rangle$};
  \node (1m) at (2,-1) {$P_3\{1\}$};
  \node (2) at (4,0) {$P_3\langle-2\rangle$};
  \draw [->] (0) -- (1);
  \draw [->] (1) -- (2);
  \draw [->] (1m) -- (2);
\end{tikzpicture}
\]

Because of the low dimension of the $\Hom$-spaces, the choice in the differentials will be very limited. For this reason, we'll often neglect to specify a choice for the differentials in a complex (as we did above).

A key message I'm trying to convey is that such complexes are easy to work with: the projectives $P_i$'s have explicit bases, have low dimension, we have explicit bases for the spaces of morphisms, and everything comes down to linear algebra.

\subsection{Spherical twists}

Now that we have a category that lifts up the vector space we care about, the next step is to define auto-equivalences on the category that will lift up the braid group action. This is done by using the notion of \emph{spherical twists} in a Calabi-Yau category, but in our case, thanks to adjunctness, this can be made very simple. One can find more details for example in~\cite{SeidelThomas}.

\begin{definition}
  Let $\Sigma_i$ and $\Sigma_i^{-1}$ be the following complexes of $A_{\Gamma}$-bimodules:
  \begin{align}
    \Sigma_i&= P_i\otimes_{\C} Q_i \xrightarrow{f} \uwave{A_{\Gamma}} \\
    \Sigma_i^{-1}&= \uwave{A_\Gamma} \xrightarrow{g} P_i\otimes_\C Q_i\langle -2\rangle 
  \end{align}
  where the terms in homological degree zero are underlined with waves, $f$ is defined by $f(e_i\otimes e_i)=e_i$ and : 
  \[
  g(1)=x_i\otimes e_i + \sum_{j\text{---}i} (j|i)\otimes (i|j) + e_i\otimes x_i
  \]
\end{definition}

$\Sigma_i$ and $\Sigma_i^{-1}$ are complexes of $A_\Gamma$-bimodules which are not projective as bimodules, but are projective as left and right modules. In particular, they induce autoequivalences on $\calC_\Gamma$ by tensor product.

\begin{lemma}
  We have the following homotopy equivalences:
  \begin{gather}
    \Sigma_i\otimes_{A_\Gamma}\Sigma_i^{-1}\otimes_{A_\Gamma}\simeq \Sigma_i^{-1}\otimes \Sigma_i\simeq A_\Gamma \\
    \Sigma_i\otimes_{A_\Gamma}\Sigma_j\otimes_{A_\Gamma}\Sigma_i\simeq \Sigma_j\otimes_{A_\Gamma}\Sigma_i\otimes_{A_\Gamma}\Sigma_j\; \text{if}\;i\text{---}j \\
    \Sigma_i\otimes_{A_\Gamma}\Sigma_j\simeq \Sigma_j\otimes_{A_\Gamma}\Sigma_i \text{if} \; i \; \text{and} \;j\;\text{are unrelated}.
  \end{gather}
\end{lemma}

\begin{proof}[Sketched]
  The proof is essentially an explicit computation. The underlying idea is that if one isolates the terms $\mathcal{U}_i=P_i\otimes Q_i$, then one can check that they satisfy Temperley-Lieb relations.

  Let's do the invertibility in details. For simplicity, let's assume that we have an $A_n$-type configuration, at least locally, with $i=2$.

  \begin{gather*}
    (P_iQ_i\rightarrow A)\otimes (A\rightarrow P_iQ_i\langle -2\rangle)=
    \begin{tikzpicture}[anchorbase]
      \node (0) at (0,0) {\small $P_iQ_i$};
      \node (1+) at (3,1) {$A$};
      \node (1-) at (3,-1) {\small $P_iQ_iP_iQ_i\langle -2\rangle$};
      \node (2) at (6,0) {\small $P_iQ\langle-2\rangle$};
      \draw [->] (0)  to node [midway,above, rotate=18] {\tiny $a\otimes b\rightarrow ab $} (1+);
      \draw [->] (0) to node [midway, below,rotate=-18] {\tiny \begin{minipage}{2cm} $a\otimes b\rightarrow$ \\ $a\otimes b\otimes c$\end{minipage}}(1-);
      \draw [->] (1+) to node [midway, above, rotate=-18] {\tiny $1\rightarrow c$} (2);
      \draw [->] (1-) to node [midway, below,xshift=.5cm, rotate=18] {\tiny \begin{minipage}{1.5cm}$a\otimes b\otimes c \otimes d$\\ $\rightarrow (a\otimes b)(cd)$\end{minipage}} (2);
    \end{tikzpicture}
  \end{gather*}    

  Above $c=x_2\otimes e_2+(1|2)\otimes (2|1) + (3|2) \otimes (2|3) + 1\otimes x_2$. Then we use $P_iQ_i\simeq \C\oplus \C\langle 2\rangle$ to get:
  
\begin{gather*}
    (P_iQ_i\rightarrow A)\otimes (A\rightarrow P_iQ_i\langle -2\rangle)=
    \begin{tikzpicture}[anchorbase]
      \node (0) at (0,0) {\small $P_iQ_i$};
      \node (1+) at (3,1) {$A$};
      \node (1-+) at (3,-1) {\small $P_iQ_i\langle -2\rangle$};
      \node (1--) at (3,-2) {\small $P_iQ_i$};
      \node (2) at (6,0) {\small $P_iQ\langle-2\rangle$};
      \draw [->] (0)  to node [midway,above, rotate=18] {\tiny $a\otimes b\rightarrow ab $} (1+);
      \draw [->] (0) to node [midway, xshift=.5cm, yshift=.2cm,rotate=-18] {\tiny \begin{minipage}{2cm} $a\otimes b\rightarrow$ \\ $e_2^*(b\otimes c_1)(a\otimes c_2)$\end{minipage}}(1-+);
      \draw [->] (0) to node [midway, below,rotate=-32] {\tiny \begin{minipage}{2cm} $a\otimes b\rightarrow$ \\ $x_2^*(b\otimes c_1)(a\otimes c_2)$\end{minipage}}(1--);
      \draw [->] (1+) to node [midway, above, rotate=-18] {\tiny $1\rightarrow c$} (2);
      \draw [->] (1-+) to node [midway, above, rotate=18] {\tiny $\id$} (2);
      \draw [->] (1--) to node [midway, below, rotate=32] {\tiny $a \otimes d\rightarrow a\otimes dx_2$} (2);
    \end{tikzpicture}
\end{gather*}

Above we've used $c=\sum c_1\otimes c_2$. Notice that the bottom-most map on the left has the following effect on the generator:
\[
e_i\otimes e_i\rightarrow e_i\otimes e_i
\]
so this is just an identity.

Running a Gaussian elimination along both identity maps brings us back to just $A$.  
\end{proof}

\begin{definition}
$\B(\Gamma)$ acts on on $\calC_\Gamma$ by sending $\sigma_i$ to $\Sigma_i\otimes -$ and $\sigma_i^{-1}$ to $\Sigma_i^{-1}\otimes -$.
\end{definition}

The effect on projectives is easy to compute.

\begin{lemma}
\begin{equation} \label{eq:actionsigmai}
\sigma_i(P_j)=\Sigma_i\otimes_{A_{\Gamma}}(P_j)=
\begin{cases}
  P_j\langle2\rangle\{1\}[-1] \;\text{if}\; i=j; \\
  P_i\langle 1\rangle \rightarrow P_j\;\text{if}\;i\rightarrow j;\\
    P_i\langle 1\rangle\{1\} \rightarrow P_j\;\text{if}\;i\leftarrow j;\\
  P_j\;\text{otherwise}.
\end{cases}
\end{equation}
Above, on the second and third lines, $P_j$ lies in homological degree zero.

\begin{equation} \label{eq:actionsigmaiinv}
\sigma_i^{-1}(P_j)=\Sigma_i^{-1}\otimes_{A_{\Gamma}}(P_j)=
\begin{cases}
  P_j\langle-2\rangle\{-1\}[1] \;\text{if}\; i=j; \\
  P_j\rightarrow P_i\langle -1\rangle\{-1\} \;\text{if}\;i \rightarrow j;\\
  P_j\rightarrow P_i\langle -1\rangle \;\text{if}\;i\leftarrow j;\\
  P_j\;\text{otherwise}.
\end{cases}
\end{equation}
\end{lemma}

This directly implies the following result.

\begin{corollary}
  The Khovanov-Seidel categorical representation decategorifies on the Burau representation.
\end{corollary}

\begin{theorem}[\cite{RouquierZimmerman,KhS,BravThomas,Riche, IshiiUedaUehara,GTW, Jensen}] \label{thm:KhSfaithfulness}
  If $\Gamma$ is of finite type or affine type $A$, and if $\beta \in \B(\Gamma)$ acts by the identity on $\calC$, then $\beta=\id$ in $\B(\Gamma)$.
\end{theorem}

In type $A$, the proof from Khovanov and Seidel~\cite{KhS} follows from Lemma~\ref{lem:inter} and geometric arguments (in particular Theorem~\ref{th:IdLeavesInvariant} that will be stated later). This argument is extended to affine type $A$ in \cite{GTW}. Brav-Thomas' proof in type ADE~\cite{BravThomas} relies on Garside theory. One uses a Garside normal form, reduces to lifts from the Weyl group, and for those one identifies descents from lifted projectives. We'll get back to these in the next section.

\begin{problem}
  What about faithfulness in general? This is closely related to the faithfulness of the action on Soergel bimodules~\cite{Rouquier_conjecture}.
\end{problem}

 \section{Curves and spherical objects}

In type $A$, a decent part of the analysis of braid groups relies on their action on curves in the disk with marked points. This interpretation naturally extends to mapping class groups of surfaces (generalizing to the curve complex for example, and Teichmüller theory, see~\cite{FarbMargalit}), but this point of view fails to extend to Artin-Tits groups, as there is no underlying surface in general. In type $A$, curves in the disk have a translation in $\calC_\Gamma$ as spherical objects. Since these constructions do extend for general families of Artin-Tits groups, it is very tempting to conjecture that they should serve as a replacement for the geometric strategies used to study mapping class groups when the geometry disappears. Outside of the range of cases covered by Theorem~\ref{thm:KhSfaithfulness}, one should either first prove faithfulness, or state results for the group of autoequivalences of $\calC_\Gamma$. But even in type $D$ and $E$ (and sometimes even in type $A$), I'll highlight a few questions that are still open and, I hope, might benefit from this categorical perspectives. We'll in particular discuss questions related to orderability and the Haagerup property.

\subsection{Action on curves}

The braid group in type $A$ can appear from several perspectives. One of them makes it the mapping class group of a disk with $n+1$ unordered marked points. Recall that mapping class groups are defined as follows, for $\Sigma$ a surface:
\[
MCG(\Sigma)=\frac{\mathrm{Diff}^+(\Sigma)}{Isotopy}
\]
Here the marked points are fixed globally, but the boundary component is fixed pointwise.

The corresponding action of $\B(A_n)$ on $\D^2$ will play an important role in what follows. A generator $\sigma_i$ acts on the punctured disk by the half Dehn twist along the segment joining the $i$-th and $i+1$-st points. Figure~\ref{fig:DehnTwist} illustrates the process.

\begin{figure}[!h]
\[
  \begin{tikzpicture}[anchorbase,scale=.5]
    \draw (0,0) ellipse (3 and 2);
    \node (1) at (-2,0) {$\bullet$};
    \node (2) at (-1,0) {$\bullet$};
    \node (3) at (-0,0) {$\bullet$};
    \node (4) at (1,0) {$\bullet$};
    \node (5) at (2,0) {$\bullet$};
    \draw [->, red] (-.9,-.1) to [out=-60,in=-120] (-.1,-.1);
    \draw [->, red] (-.1,.1) to [out=120,in=60] (-.9,.1);
  \end{tikzpicture}
  \]
  \caption{Action of $\sigma_2$}
  \label{fig:DehnTwist}
\end{figure}
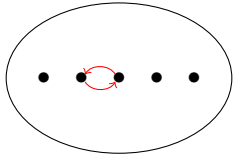

From there one can focus on the action on curves drawn in the punctured disk. The family of curves that appear on the left hand side in Figure~\ref{fig:curves1} will play a key role in what follows. On the right hand side is depicted the image of the same family under the braid $\sigma_1\sigma_2^{-1}$.

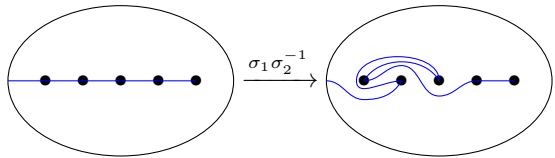
\begin{figure}
  \[
  \begin{tikzpicture}[anchorbase,scale=.5]
    \draw (0,0) ellipse (3 and 2);
    \node (1) at (-2,0) {$\bullet$};
    \node (2) at (-1,0) {$\bullet$};
    \node (3) at (-0,0) {$\bullet$};
    \node (4) at (1,0) {$\bullet$};
    \node (5) at (2,0) {$\bullet$};
    \draw [blue] (-3,0) -- (1.center);
    \draw [blue](1.center) -- (2.center);
    \draw [blue](2.center) -- (3.center);
    \draw [blue](3.center) -- (4.center);
    \draw [blue](4.center) -- (5.center);
  \end{tikzpicture}
  \xrightarrow{\sigma_1\sigma_2^{-1}}
  \begin{tikzpicture}[anchorbase,scale=.5]
    \draw (0,0) ellipse (3 and 2);
    \node (1) at (-2,0) {$\bullet$};
    \node (2) at (-1,0) {$\bullet$};
    \node (3) at (-0,0) {$\bullet$};
    \node (4) at (1,0) {$\bullet$};
    \node (5) at (2,0) {$\bullet$};
    \draw [blue](-3,0) to [out=0,in=180] (-2,-.5) to [out=0,in=-120] (2.center);
    \draw [blue](2.center) to [out=180,in=-90] (-2.2,0) to [out=90, in=90] (3.center);
    \draw [blue](3.center) to [out=120,in=60] (1.center);
    \draw [blue](1.center) to [out=0,in=180] (2.north) to [out=0,in=180] (3.south) to [out=0,in=180] (4.center);
    \draw [blue](4.center) -- (5.center);
  \end{tikzpicture}
\]
\caption{Standard curves and braid action}
\label{fig:curves1}
\end{figure}

A key result is the faithfulness of this action on curves, which is a classical result on mapping class groups~(see for example Lemma 2.1 combined with Proposition 1.11, that will need to be upgraded for collections of curves, in~\cite{FarbMargalit}). This result is a key step in Khovanov and Seidel's faithfulness proof.

\begin{theorem} \label{th:IdLeavesInvariant}
  A braid that leaves all curves invariant is the identity.
\end{theorem}

Now might be a good time for a little {\it apart\'e} about orderability of braid and Artin-Tits groups, as this is closely related to the action on curves.

\subsection{Ordering braids}

\begin{definition}
  $G$ is left-orderable if there exists a total order $\leq$ on $G$ that is left invariant:
  \[
\forall a,b,c\in G,\quad b\leq c \;\Rightarrow\; ab\leq ac.
  \]
\end{definition}

There are analogous notions of right-orderable or bi-orderable groups.

Braids were proved to be orderable by Dehornoy in the early 90's~\cite{Dehornoy} (see~\cite{Kassel_bourbaki} for a survey). His original construction relies on set theory outside of the Zermelo-Fraenkel axioms. This is a quite surprising connection, but more down-to-earth reinterpretations were found soon after~\cite{FGRRW}. Close to our interest is the fact that one can use curves to describe this order.

\begin{theorem}
  The following defines a left-invariant order. Given two braids $\alpha$ and $\beta$, consider the images under $\alpha$ and $\beta$ of the system of curves from Figure~\ref{fig:curves1}, and put them in minimal position. Starting from the left, locate the first curve that admits different images under $\alpha$ and $\beta$. If the image under $\alpha$ first goes lower than the other one, then $\alpha>\beta$. Otherwise $\beta>\alpha$.
\end{theorem}

\begin{example}
  From Figure~\ref{fig:curves1}, one can see that $\sigma_1\sigma_2^{-1}$ is positive (greater than $\id$).
\end{example}

Amongst consequences of orderability, it implies that the group satisfies the Kaplansky idempotent conjecture: the only idempotent in the group algebra $\C[G]$ is $\id$. For representation theorists that tend to look at braids from the Hecke algebra (recall that it is still an open question to show that braids inject in the Hecke algebra), where we have tons of idempotents, this is a bit of a strange behavior!

As far as I know, orderability has been proven for a very restricted family of braid groups.
\begin{theorem}[\cite{Dehornoy},\cite{MulhollandRolfsen} building upon \cite{CrispParis}]
  Artin-Tits groups are orderable in types $A$ and $D$.
\end{theorem}

In type $D$, the proof uses Crisp and Paris decomposition: $\B(D_n)\simeq \B(A_{n-1})\ltimes F_{n-1}$, where $F_{n-1}$ is the free group on $n$ generators. The type $A$ part is obtained through the following morphism:
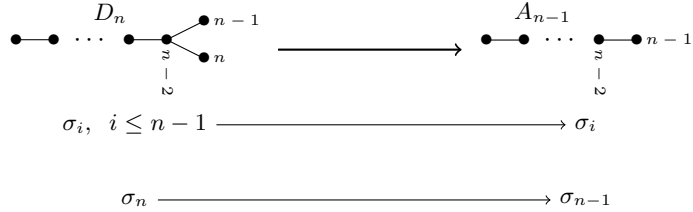
\[
\begin{tikzpicture}[anchorbase]
  \node (D) at (0,0) {
    \begin{tikzpicture}[anchorbase,scale=.5]
      \node at (3.5,.7) {\small $D_n$};
      \node (1) at (1,0) {$\bullet$};
      \node (2) at (2,0) {$\bullet$};
      \node (d) at (3,0) {$\cdots$};
      \node (n-3) at (4,0) {$\bullet$};
      \node (n-2) at (5,0) {$\bullet$};
      \node (n-1) at (6,.5) {$\bullet$};
      \node (n) at (6,-.5) {$\bullet$};
      \draw (1.center) -- (2.center);
      \draw (n-3.center) -- (n-2.center);
      \draw (n-2.center) -- (n-1.center);      
      \draw (n-2.center) -- (n.center);
      \node [rotate=-90,right] at (n-2) {\tiny $n-2$};
      \node [right] at (n-1) {\tiny $n-1$};
      \node [right] at (n) {\tiny $n$};
      \end{tikzpicture}
  };
  \node (A) at (6,0) {
      \begin{tikzpicture}[anchorbase,scale=.5]
      \node at (2.5,.7) {\small $A_{n-1}$};
      \node (1) at (1,0) {$\bullet$};
      \node (2) at (2,0) {$\bullet$};
      \node (d) at (3,0) {$\cdots$};
      \node (n-2) at (4,0) {$\bullet$};
      \node (n-1) at (5,0) {$\bullet$};
      \draw (1.center) -- (2.center);
      \draw (n-2.center) -- (n-1.center);
      \node [rotate=-90,right] at (n-2) {\tiny $n-2$};
      \node [right] at (n-1) {\tiny $n-1$};
    \end{tikzpicture}
  };
  \draw [thick, ->] (D) -- (A);
  \node (genD) at (0,-1) {\small $\sigma_i, \;\; i \leq n-1$};
  \node (genDn) at (0,-2) {\small $\sigma_n$};
  \node (genA) at (6,-1) {\small $\sigma_i$};
  \node (genAn-1) at (6,-2) {\small $\sigma_{n-1}$};
  \draw [->, opacity=.7] (genD) -- (genA);
  \draw [->, opacity=.7] (genDn) -- (genAn-1);
\end{tikzpicture}
\]

In type $D$ again, one can use Perron-Vannier's injection of the Artin-Tits group into the MCG of a surface~\cite{PerronVannier}. Then orderability follows from Rourke and Wiest~\cite{RourkeWiest_order}, also based on curve considerations.

\begin{problem}
  Prove orderability in full generality, starting from type $E$. Even in type $D$, can one give an interpretation of the order in terms of Khovanov-Seidel's category?
\end{problem}

Approaches with closer ties to surfaces might play an interesting role here, such as~\cite{Qiu_decorated} and subsequent work.

\subsection{Curves and spherical objects}

We now go back to our goal of understanding how curves in the disks and zigzag modules relate.

\subsubsection{Correspondence}

Let us consider the type $A$ case. From a curve in the punctured disk, one can build a complex in $\calC_\Gamma$ as follows.

First we pull tight the curve, so it is in a minimal position. We can associate to each segment $(i,i+1)$ the projective indecomposable $P_i$. Starting from one end, we arbitrarily choose a path length and homological degree for the corresponding $P_i$, and then follow the curve and inductively build the corresponding complex with following local glueing rules (here we have only indicated the path-length grading):
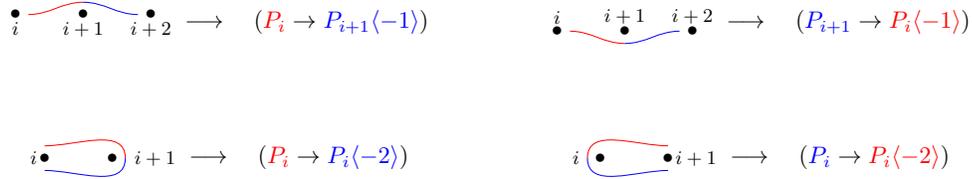
\[
\begin{tikzpicture}[anchorbase,scale=.9,every node/.style={scale=0.9}]
  \node (A1) at (0,0) {
$    \begin{tikzpicture}[anchorbase]
      \node (A11) at (1,0) {$\bullet$};
      \node (A12) at (2,0) {$\bullet$};
      \node (A13) at (3,0) {$\bullet$};
      \node [below] at (A11) {\small $i$};
      \node [below] at (A12) {\small $i+1$};
      \node [below] at (A13) {\small $i+2$};
      \draw [red] (A11) to [out=0,in=180] (A12.north);
      \draw [blue] (A12.north) to [out=0,in=180] (A13);
    \end{tikzpicture}
    \longrightarrow \quad \left(\textcolor{red}{P_i}\rightarrow \textcolor{blue}{P_{i+1}\langle -1\rangle}\right)
    $
  };
  \node (A2) at (8,0) {
$    \begin{tikzpicture}[anchorbase]
      \node (A21) at (1,0) {$\bullet$};
      \node (A22) at (2,0) {$\bullet$};
      \node (A23) at (3,0) {$\bullet$};
      \node [above] at (A21) {\small $i$};
      \node [above] at (A22) {\small $i+1$};
      \node [above] at (A23) {\small $i+2$};
      \draw [red] (A21) to [out=0,in=180] (A22.south);
      \draw [blue] (A22.south) to [out=0,in=180] (A23);
    \end{tikzpicture}
    \longrightarrow \quad \left(\textcolor{blue}{P_{i+1}}\rightarrow \textcolor{red}{P_{i}\langle -1\rangle}\right)
    $
  };
  \node (B1) at (0,-2) {
$    \begin{tikzpicture}[anchorbase]
      \node (B11) at (1,0) {$\bullet$};
      \node (B12) at (2,0) {$\bullet$};
      \node [left] at (B11) {\small $i$};
      \node [xshift=.7cm] at (B12) {\small $i+1$};
      \draw [red] (B11.north) to [out=0,in=90] (B12.east);
      \draw [blue] (B12.east) to [out=-90,in=0] (B11.south);
    \end{tikzpicture}
    \longrightarrow \quad \left(\textcolor{red}{P_i}\rightarrow \textcolor{blue}{P_{i}\langle -2\rangle}\right)
    $
  };
  \node (B2) at (8,-2) {
$    \begin{tikzpicture}[anchorbase]
      \node (B21) at (1,0) {$\bullet$};
      \node (B22) at (2,0) {$\bullet$};
      \node [xshift=-.4cm] at (B21) {\small $i$};
      \node [right] at (B22) {\small $i+1$};
      \draw [red] (B21.west) to [out=90,in=180] (B12.north);
      \draw [blue] (B22.south) to [out=180,in=-90] (B11.west);
    \end{tikzpicture}
    \longrightarrow \quad \left(\textcolor{blue}{P_i}\rightarrow \textcolor{red}{P_{i}\langle -2\rangle}\right)
    $
  };
\end{tikzpicture}
\]

The following result makes the type $A$ situation especially easy to handle.

\begin{theorem}[\cite{IshiiUedaUehara, IshiiUehara,AdachiMizuniYang,BapatDeopurkarLicata_spherical}]
  This is a $1:1$ map, establishing a bijection between curves and spherical objects up to shift.
\end{theorem}

One can then check the following.
\begin{lemma}
  The action by half Dehn twists on curves and the Khovanov-Seidel action on $\calC_{A_n}$ are compatible.
\end{lemma}
\begin{proof}[Sketched]
  The proof is a case by case analysis. Let's try an easy example:

  
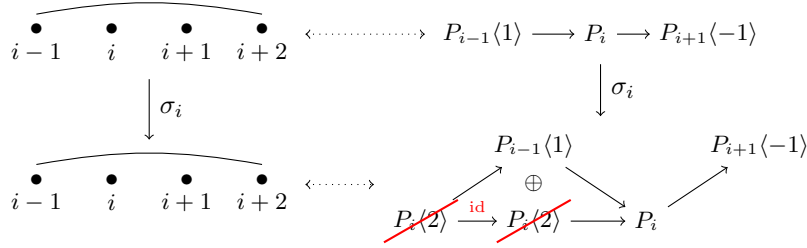
\[
  \begin{tikzpicture}[anchorbase]
    \node (A) at (0,0) {
      \begin{tikzpicture}[anchorbase]
        \node (i-1) at (0,0) {$\bullet$};
        \node [yshift=-.3cm] at (i-1) {\small $i-1$};
        \node (i) at (1,0) {$\bullet$};
        \node [yshift=-.3cm] at (i) {\small $i$};
        \node (i+1) at (2,0) {$\bullet$};
        \node [yshift=-.3cm] at (i+1) {\small $i+1$};
        \node (i+2) at (3,0) {$\bullet$};
        \node [yshift=-.3cm] at (i+2) {\small $i+2$};
        \draw (i-1.north) to [out=10,in=170] (i+2.north);
      \end{tikzpicture}
    };
    \node (B) at (0,-2) {
            \begin{tikzpicture}[anchorbase]
        \node (i-1) at (0,0) {$\bullet$};
        \node [yshift=-.3cm] at (i-1) {\small $i-1$};
        \node (i) at (1,0) {$\bullet$};
        \node [yshift=-.3cm] at (i) {\small $i$};
        \node (i+1) at (2,0) {$\bullet$};
        \node [yshift=-.3cm] at (i+1) {\small $i+1$};
        \node (i+2) at (3,0) {$\bullet$};
        \node [yshift=-.3cm] at (i+2) {\small $i+2$};
        \draw (i-1.north) to [out=10,in=170] (i+2.north);
      \end{tikzpicture}
    };
    \draw[->] (A) -- (B) node[midway,xshift=.3cm] {$\sigma_i$};
    \node (C) at (6,0) {
      \begin{tikzpicture}[anchorbase]
        \node (O1) at (0,0) {\small $P_{i-1}\langle 1\rangle$};
        \node (O2) at (1.5,0) {\small $P_{i}$};
        \node (O3) at (3,0) {\small $P_{i+1}\langle -1\rangle$};
        \draw [->] (O1) -- (O2);
        \draw [->] (O2) -- (O3);
      \end{tikzpicture}
    };
    \draw [<->,dotted] (A) -- (C);
        \node (D) at (6,-2) {
      \begin{tikzpicture}[anchorbase]
        \node (O1) at (0,0) {\small $P_{i-1}\langle 1\rangle$};
        \node (O3) at (3,0) {\small $P_{i+1}\langle -1\rangle$};
        \node (sO1) at (-1.5,-1) {\small $P_i\langle 2\rangle$};
        \node (sO2) at (0,-1) {\small $P_i\langle 2\rangle$};
        \node (sO3) at (1.5,-1) {\small $P_i$};
        \node at (0,-.5) {\small $\oplus$};
        \draw [->] (O1) -- (sO3);
        \draw [->] (sO1) -- (O1);
        \draw [->] (sO3) -- (O3);
        \draw [->] (sO1) -- (sO2) node [midway,above,red] {\tiny $\id$};
        \draw [->] (sO2) -- (sO3);
        \draw [red,thick] (sO1.south west) -- (sO1.north east);
        \draw [red,thick] (sO2.south west) -- (sO2.north east);
      \end{tikzpicture}
    };
    \draw[->] (C) -- (D) node[midway,xshift=.3cm] {$\sigma_i$};
    \draw [<->,dotted] (B) -- (D);
  \end{tikzpicture}
  \]

\end{proof}

There also is a correspondence between the geometric intersection between curves and the dimension of the $\Hom$-space.

\begin{lemma}\label{lem:inter}
Let $C$ and $D$ be spherical objects in $\calC_{A_n}$. Then $\dim(\Hom_{\calC}(C,D))$ is given by twice the number of intersection points between curve representatives in minimal position. Intersections at origin count just for one.
\end{lemma}
\begin{proof}
This is an explicit count (see \cite{KhS}[Section 4c]).
\end{proof}

\subsubsection{Spherical twists, again} \label{sec:sphTwAgain}

We had seen that the spherical twist on $P_i$ is the mapping cone of the multiplication map from $P_i\otimes Q_i$ to $A_\Gamma$. Given a curve $c$, it can be associated an object $C=\beta P_i$ for some $i$ and some braid $\beta$. Denote $C^\vee=Q_i \beta$ (for the right action). Then the half Dehn twist on $c$ matches the spherical twist on $C$:
\[
T_C=C\otimes C^\vee \rightarrow A_\Gamma,
\]
and this doesn't depend on $\beta$ nor $i$.

With these constructions in hand, we can revisit Bigelow's argument to find elements in the kernel.

\begin{proof}[Proof of Theorem~\ref{thm:BurauKernel}]
  Bigelow's proof consists in exhibiting a pair of curves in the punctured plane that do intersect, but so that the corresponding elements in $V$ have trivial pairing. The search for such pairs is done by computer.

  Let's look at the construction of the counterexample in the case where both curves relate punctures in the interior of the disk (Bigelow actually allows one of them to have an end on the boundary, which we can also understand categorically, but the construction is slightly trickier). So we have $\gamma_1$ and $\gamma_2$ two curves with trivial pairing, which means that there are associated spherical objects $A_1$ and $A_2$:
  \[
A_2^{\vee}\otimes A_1\neq 0\quad \text{but} \quad \qdim(A_2^{\vee}\otimes A_1)=0
\]
Looking at the spherical twists:
\[
T_{A_1}\otimes T_{A_2}=
    \begin{tikzpicture}[anchorbase]
      \node (0) at (0,0) {\small $\left(A_1\otimes A_2^{\vee}\right)^{\oplus_{A_1^\vee\otimes A_2}}$};
      \node (1+) at (3,1) {$A_1\otimes A_1^\vee$};
      \node (1-) at (3,-1) {\small $A_2\otimes A_2^\vee$};
      \node (2) at (6,0) {\small $A$};
      \draw [->] (0)  to node [midway,above, rotate=18] {\tiny} (1+);
      \draw [->] (0) to node [midway, below,rotate=-18] {\tiny}(1-);
      \draw [->] (1+) to node [midway, above, rotate=-18] {\tiny} (2);
      \draw [->] (1-) to node [midway, below,xshift=.5cm, rotate=18] {\tiny } (2);
    \end{tikzpicture}
\]
\[
T_{A_2}\otimes T_{A_1}=
    \begin{tikzpicture}[anchorbase]
      \node (0) at (0,0) {\small $\left(A_2\otimes A_1^{\vee}\right)^{\oplus_{A_2^\vee\otimes A_1}}$};
      \node (1+) at (3,1) {$A_1\otimes A_1^\vee$};
      \node (1-) at (3,-1) {\small $A_2\otimes A_2^\vee$};
      \node (2) at (6,0) {\small $A$};
      \draw [->] (0)  to node [midway,above, rotate=18] {\tiny} (1+);
      \draw [->] (0) to node [midway, below,rotate=-18] {\tiny}(1-);
      \draw [->] (1+) to node [midway, above, rotate=-18] {\tiny} (2);
      \draw [->] (1-) to node [midway, below,xshift=.5cm, rotate=18] {\tiny } (2);
    \end{tikzpicture}
    \]

    The two twists should differ in general, but they decategorify to the same thing.
\end{proof}

\subsubsection{Orderability}

Here I'll try to reinterpret the curve condition for orderability in terms of spherical objects in Khovanov-Seidel's category. Somehow, the trouble is that in type $A$, they are many ways to do so, but none of them seems to extend on the nose to an orderability criterion outside of type $A$. Still, there is one involving exceptional collections that I like better than others.

We will first extend a bit the algebraic situation so as to be able to have curves touching a point on the boundary, just like in Figure~\ref{fig:curves1}. To do so, embed the type $A_n$ Dynkin diagram into one of type $A_{n+1}$, with the extra node labeled $0$ placed on the left, and add a ``basing'' quadratic relation in the zig-zag algebra:
\[
\begin{tikzpicture}[anchorbase]
  \node (0) at (0,0) {$\bullet$};
  \node [yshift=.3cm] at (0) {\small $0$};
  \node (1) at (1,0) {$\bullet$};
  \node [yshift=.3cm] at (1) {\small $1$};
  \node (2) at (2,0) {$\bullet$};
  \node [yshift=.3cm] at (2) {\small $2$};
  \node (3) at (3,0) {$\bullet$};
  \node [yshift=.3cm] at (3) {\small $3$};
  \draw (0.center) -- (1.center);
  \draw (1.center) -- (2.center);
  \draw (2.center) -- (3.center);
  \node [xshift=.5cm] at (3) {$\cdots$};
\end{tikzpicture}
\qquad \text{with} \quad
\begin{tikzpicture}[anchorbase]
  \node (0) at (0,0) {$\bullet$};
  \node [yshift=.3cm] at (0) {\small $0$};
  \node [yshift=-.3cm] at (0) {\vphantom{$0$}};
  \node (1) at (1,0) {$\bullet$};
  \node [yshift=.3cm] at (1) {\small $1$};
  \draw [->] (.15,.05) to [out=20,in=140] (.9,0) to [out=-140,in=-20] (.15,-.05); 
\end{tikzpicture}
=0
\]

We can then consider the following {\it exceptional collection} of objects (the terminology comes from algebraic geometry, see for example the introduction in~\cite{Bondal_ex}):
\[
(P_0,P_0\rightarrow P_1\langle -1\rangle,P_0\rightarrow P_1\langle -1\rangle \rightarrow P_2\langle -2\rangle,\cdots)
\quad \leftrightarrow \quad
\begin{tikzpicture}[anchorbase,scale=.8]
  \draw (0,0) ellipse (2 and 1);
  \node (0) at (-2,0) {};
  \node (1) at (-1,0) {$\bullet$};
  \node (2) at (0,0) {$\bullet$};
  \node (3) at (1,0) {$\bullet$};
  \node at (1.5,0) {$\cdots$};
  \draw [semithick] (0.center) -- (1.center);
  \draw [semithick] (0.center) to [out=20,in=160]  (2.center);
  \draw [semithick] (0.center) to [out=40,in=140]  (3.center);  
\end{tikzpicture}
\]

The key points here are that these objects, except of being spherical, are exceptional, with $\END$-space of dimension $1$ concentrated in degree zero, and furthermore there is a breach of symmetry:
\begin{gather*}
  \HOM(P_0,P_0\rightarrow P_1\langle -1\rangle)=\C \id_{P_0} \\
  \HOM(P_0\rightarrow P_1\langle -1\rangle,P_0)=\{0\}
\end{gather*}

Now, let's assume that $\beta$ moves the curve corresponding to $P_0$. Then either $\beta P_0$ will map into $P_0$, or it will be mapped into by $P_0$. Indeed, $\beta P_0$ contains a single copy of $P_0$, and looking at the curve/object correspondence, it appears that exactly one of the two maps between $P_0$ and $\beta P_0$ given by an identity of $P_0$ (one direction or the other) will define a non-trivial chain map. Whether we fall in one or the other situation exactly captures the fact that the curve goes up or down. See the LHS of Figure~\ref{fig:Excep_order} for an example.

In type $D$, one can try to copy this construction. But it appears that there are objects (spherical or exceptional) that are moved but seem to go neither up nor down: see the RHS of Figure~\ref{fig:Excep_order}.

\begin{figure}[!h]
\[
\begin{tikzpicture}[anchorbase]
  \node (disk) at (0,0) {
  \begin{tikzpicture}[anchorbase,scale=.6]
  \draw (0,0) ellipse (2 and 1);
  \node (0) at (-2,0) {};
  \node (1) at (-1,0) {$\bullet$};
  \node (2) at (0,0) {$\bullet$};
  \node (3) at (1,0) {$\bullet$};
  \draw [semithick,blue] (0.center) -- (1.center);
  \draw [semithick,red] (0.center) to [out=20,in=180]  (1.north) to [out=0,in=180] (2.south) to [out=0,in=-160] (3.center);
\end{tikzpicture}
  };
  \node (excepA) at (0,-3) {
    \begin{tikzpicture}[anchorbase]
      \node [blue] (P0t) at (0,0) {\small $P_0$};
      \node [red] (P0) at (0,-2) {\small $P_0$};
      \node [red] (P1) at (2,-2.5) {\small $P_1\langle -1\rangle$};
      \node [red] (P2) at (0,-3) {\small $P_2$};
      \node [red] at (0,-2.5) {\small $\oplus$};
      \draw [red,->] (P0) -- (P1);
      \draw [red,->] (P2) -- (P1);
      \draw [green,->] (P0.north)+(.1,0) -- (.1,-.3) node [midway,xshift=.2cm] {\small $\id$};
      \draw [opacity=.4,->] (-.2,-.3) -- (-.2,-1.7)  node [midway, xshift=-.2cm] (ID) {\small $\id$}   ;
      \draw [thick] (ID.south west) -- (ID.north east);
      \draw [thick] (ID.north west) -- (ID.south east);
    \end{tikzpicture}
  };
  \node (DynkinD) at (6,0) {
    \begin{tikzpicture}[anchorbase,scale=.5]
      \node (1) at (4,0) {$\bullet$};
      \node (2) at (5,0) {$\bullet$};
      \node (3) at (6,.5) {$\bullet$};
      \node (4) at (6,-.5) {$\bullet$};
      \draw (1.center) -- (2.center);
      \draw (2.center) -- (3.center);
      \draw (2.center) -- (4.center);      
    \end{tikzpicture}
  };
  \node (excepD) at (6,-3) {
    \begin{tikzpicture}[anchorbase]
      \node [blue] (P2t) at (0,0) {\small $P_2$};
      \node [red] (P3) at (-2,-2) {\small $P_3\langle 1\rangle$};
      \node [red] (P2) at (0,-2) {\small $P_2$};
      \node [red] (P4) at (2,-2) {\small $P_4\langle-1\rangle$};
      \draw [red,->] (P3) -- (P2);
      \draw [red,->] (P2) -- (P4);
      \draw [opacity=.4,->] (-.2,-.3) -- (-.2,-1.7)  node [midway, xshift=-.2cm] (ID) {\small $\id$}   ;
      \draw [thick] (ID.south west) -- (ID.north east);
      \draw [thick] (ID.north west) -- (ID.south east);
      \draw [opacity=.4,<-] (.2,-.3) -- (.2,-1.7)  node [midway, xshift=.2cm] (ID2) {\small $\id$}   ;
      \draw [thick] (ID2.south west) -- (ID2.north east);
      \draw [thick] (ID2.north west) -- (ID2.south east);
      \draw [->] (3,-.2) to [out=-30,in=30] (3,-1.8);
      \node [rotate=90] at (3.7,-1) {\small $\sigma_3\sigma_4^{-1}$};
    \end{tikzpicture}
  };
\end{tikzpicture}
  \]
  \caption{Going up or down}
  \label{fig:Excep_order}
\end{figure}
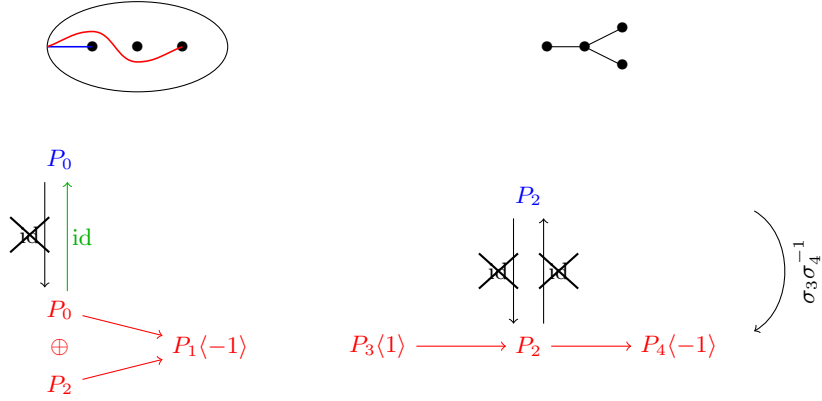

\subsection{Linear complexes}

Here we want to introduce notions that will help us slice any object into simpler pieces. This can be achieved in several ways: we'll focus on two of them. The notions we introduce here make sense in any type, but classifications results are mostly obtained in finite type.

\begin{definition}
  A complex in $\calC_\Gamma$ will be said to be linear for the path length grading if it is homotopy equivalent to one only containing terms of the kind $P_i[k]\langle -k\rangle$ for varying values of $k$ (homological and path-length shifts balance).

  A complex in $\calC_\Gamma$ will be said to be linear for the orientation grading if it is homotopy equivalent to one only containing terms of the kind $P_i[k]\{0\}$ for varying values of $k$ (orientation grading of degree zero).
\end{definition}

For the second definition, we sometimes refer to it as {\it dual} linearity. Indeed, this is related to Birman-Ko-Lee's generating set.

\begin{example}

  In type $A$, the object associated to the following curve is linear for the path-length grading:
  
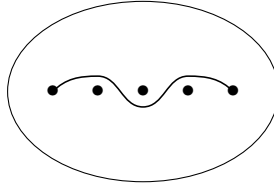
\[
  \begin{tikzpicture}[anchorbase,scale=.6]
  \draw (0,0) ellipse (3 and 2);
  \node (1) at (-2,0) {$\bullet$};
  \node (2) at (-1,0) {$\bullet$};
  \node (3) at (0,0) {$\bullet$};
  \node (4) at (1,0) {$\bullet$};
  \node (5) at (2,0) {$\bullet$};
  \draw [semithick] (1.center) to [out=45,in=180]  (2.north) to [out=0,in=180] (3.south) to [out=0,in=180] (4.north) to [out=0,in=135] (5.center);
  \end{tikzpicture}
  \]
  Linear pieces are those with no vertical tangencies.

  On the dual side, linear pieces are those that stay above all points at all time:
  
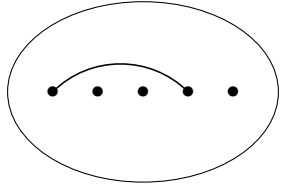
\[
  \begin{tikzpicture}[anchorbase,scale=.6]
  \draw (0,0) ellipse (3 and 2);
  \node (1) at (-2,0) {$\bullet$};
  \node (2) at (-1,0) {$\bullet$};
  \node (3) at (0,0) {$\bullet$};
  \node (4) at (1,0) {$\bullet$};
  \node (5) at (2,0) {$\bullet$};
  \draw [semithick] (1.center) to [out=45,in=135] (4.center);
  \end{tikzpicture}
  \]

\end{example}

On the dual side, there is an easy classification of linear complexes up to homological shift.

\begin{theorem}[\cite{LQ}, based on Gabriel's theorem]
  Indecomposable dual-linear spherical objects are in bijection with the set of positive roots.

  Furthermore, if $C$ is spherical, dual linear and indecomposable with corresponding root $\alpha$, then the spherical twist $T_C=T_{\alpha}$ (the ones from Section~\ref{sec:dualgen}).
\end{theorem}

\begin{example}
  In type $A_2$,
  \[
  \alpha=\alpha_1+\alpha_2 \;\leftrightarrow \;
  \begin{tikzpicture}[anchorbase]
    \node (1) at (0,0) {$\bullet$};
    \node (2) at (1,0) {$\bullet$};
    \node (3) at (2,0) {$\bullet$};
    \draw[semithick] (1.center) to [out=45,in=135] (3.center);
  \end{tikzpicture}
  \;\leftrightarrow \; P_1\rightarrow P_2\langle -1\rangle
\]
and indeed:
\[
\sigma_1\sigma_2\sigma_1^{-1}\left(P_1\rightarrow P_2\right) = \sigma_1\sigma_2(P_2) =\sigma_1\left(P_2 \{1\}\right) =P_1\{1\} \rightarrow P_2\{1\}
\]
\end{example}

One can also consider shifts of linear complexes, which have all their terms of the kind $P_i[k]\langle -k+r\rangle$ (respectively, $\{r\}$), for a fixed value of $r$. This allows to cut a complex into pieces, which is a well-defined process if one starts from a reduced complex. It is then interesting to measure how much a given braid will make linear complexes become non-linear. This is the idea behind the following definition.

\begin{definition}
  Let $\calC^0$ be the subcategory of linear complexes (for whichever grading), $\calC^r$ its $r$-shift, and $\calC^{[r,s]}$ the subcategory whose all complexes have all their terms between degree $r$ and $s$.

  A braid $\beta$ gets assigned a {\it spread} as follows:
\begin{gather*}
  sp(\beta)=r+s+1 \\
  r=min\{ l\in \N,\; \beta \calC^0\subset \calC^{\geq -l}  \} \\
   s=min\{ l\in \N,\; \beta \calC^0\subset \calC^{\leq l}  \} 
  \end{gather*}
\end{definition}

The following theorem is the main result in a joint work with Anthony Licata, showing that some metrics on the braid group have a natural categorical interpretation.

\begin{theorem}[{\cite[Theorems 4.1, 4.12]{LQ}}] \label{thm:LQ_spread}
  The classical spread equals the word length in generators from $[1,\Delta]\cup [1,\Delta]^{-1}$.

The dual spread equals the word length in generators from $[1,\gamma]\cup [1,\gamma]^{-1}$.
\end{theorem}

The key ingredient in this theorem is Garside theory: we show that one can read off the descents of a word by the top or bottom layers of its images. Indeed, if $C$ is a linear complex, it is quite clear that the only thing that can jump out of the linear part is a direct sum of $P_i$'s: this follows from the equation \eqref{eq:actionsigmai}. There is some sort of reciprocal to this result. It is easy to state for a braid which is a (positive or negative) lift from the Weyl group. To give a more precise sense to ``things that jump out'', let us first state the following definitions.

Introduce $\mathfrak{P}_i:=\bigoplus_k P_i\langle -k\rangle [k]$, and define:
\begin{gather*}
  E_r^+(X)=\{i \;:\; \Hom(\mathfrak{P}_i\langle -r\rangle ,X)\neq 0\} \\
  E_r^-(X)=\{i \;:\; \Hom(X,\mathfrak{P}_i\langle -r\rangle )\neq 0\}
  \end{gather*}

\begin{proposition}[{\cite[Proposition 4.13]{LQ}}]
  Let $\Gamma$ be of type ADE and $\beta$ be a negative lift from the Weyl group.  Then:
  $E_1^+(\beta\oplus P_r)=D_L(w)$.
\end{proposition}

This result, holding for generators of the Garside structure, spreads as follows as one reads a full positive word. 

\begin{lemma}[{\cite[Lemma 4.17]{LQ}}]
  Let $\beta=\beta_k\cdots \beta_1\in \B^+(\Gamma)$ be right-greedy. Then:
  $E^+_k\left( \beta \cdot \left( \oplus_{i\in D^+_R(\beta_1)}P_i\right)\right)=D_l(\pi(\beta_k)).$
\end{lemma}

Then one has to make sure that what jumped up won't jump down when applying the negative part of a word from a well-chosen normal form to obtain Theorem~\ref{thm:LQ_spread}.

An interesting feature here is that the same category can be used for either one of the two degrees. It is thus possible to consider both of them at once, which yielded a proof in type D of a conjecture by Digne and Gobet (proven in type $A$ and $E$ by Digne and Gobet, and reproven since in type $D$ by Baumeister and Gobet).

\begin{theorem}[\cite{DigneGobet,LQ,BaumeisterGobet}]
  \[
[1,\gamma]\subset [\Delta^{-1},\Delta]
  \]
\end{theorem}
What the above means is that any element between $1$ and $\gamma$ for the dual Garside structure can be expressed as a product of at most $2$ generators for the classical Garside structure, one positive, one negative.

\vspace{.5cm}

Let us now go back to better understanding linear complexes, and see what we can use this structure for.

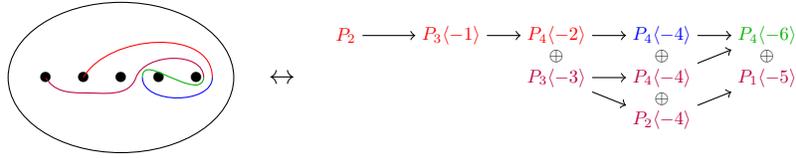
\begin{example} \label{ex:curvetocomplex}
  \[
    \begin{tikzpicture}[anchorbase,scale=.5]
    \draw (0,0) ellipse (3 and 2);
    \node (1) at (-2,0) {$\bullet$};
    \node (2) at (-1,0) {$\bullet$};
    \node (3) at (-0,0) {$\bullet$};
    \node (4) at (1,0) {$\bullet$};
    \node (5) at (2,0) {$\bullet$};
    \draw [red] (2.center) to [out=60,in=90] (5.east);
    \draw [blue] (5.east) to [out=-90,in=-90] (4.west);
    \draw [green] (4.west) to [out=90,in=180] (2,-.2) to [out=0,in=-90] (2.2,0);
    \draw [purple] (2.2,0) [out=90,in=70] to (.4,0) to [out=-110,in=0] (2.south) to [out=180,in=-60] (1.center);
    \end{tikzpicture}
    \quad \leftrightarrow
    \quad
    \begin{tikzpicture}[anchorbase,scale=.7,every node/.style={scale=0.7}]
      \node [red] (a) at (0,0) {$P_2$};
      \node [red] (b) at (2,0) {$P_3\langle -1\rangle$};
      \node [red] (c) at (4,0) {$P_4\langle -2\rangle$};
      \node [blue] (d) at (6,0) {$P_4\langle -4\rangle$};
      \node [green] (e) at (8,0) {$P_4\langle -6\rangle$};
      \node [purple] (f) at (6,-.8) {$P_4\langle -4\rangle$};
      \node [purple] (g) at (4,-.8) {$P_3\langle -3\rangle$};
      \node [purple] (h) at (6,-1.6) {$P_2\langle -4\rangle$};
      \node [purple] (i) at (8,-.8) {$P_1\langle -5\rangle$};
      \node at (4,-.4) {$\oplus$};
      \node at (6,-.4) {$\oplus$};
      \node at (6,-1.2) {$\oplus$};
      \node at (8,-.4) {$\oplus$};
      \draw [->] (a) -- (b);
      \draw [->] (b) -- (c);
      \draw [->] (c) -- (d);
      \draw [->] (d) -- (e);
      \draw [->] (f) -- (e);
      \draw [->] (g) -- (f);
      \draw [->] (g) -- (h);
      \draw [->] (h) -- (i);
    \end{tikzpicture}
    \]
    We might want to re-organize the above complex as a double complex, to highlight the shifts in linear complexes. Going to the right uses a differential of degree $1$ (namely, a map $P_i\rightarrow P_{i+1}$) while pieces of differentials going up increase the degree by $2$, through a loop. We have also indicated the global homological degree, which is supported on diagonals.
    
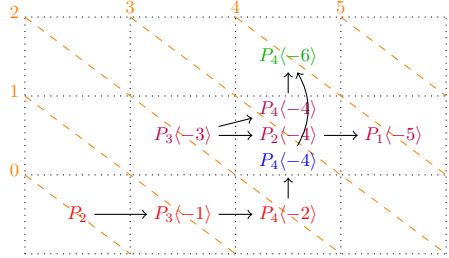
\[
    \begin{tikzpicture}[anchorbase,scale=.7,every node/.style={scale=0.7}]
      \draw [opacity=.5,dotted] (-1,-.75) -- (-1,3.75);
      \draw [opacity=.5,dotted] (1,-.75) -- (1,3.75);
      \draw [opacity=.5,dotted] (3,-.75) -- (3,3.75);
      \draw [opacity=.5,dotted] (5,-.75) -- (5,3.75);
      \draw [opacity=.5,dotted] (7,-.75) -- (7,3.75);
      \draw [opacity=.5,dotted] (-1,-.75) -- (7,-.75);
      \draw [opacity=.5,dotted] (-1,.75) -- (7,.75);
      \draw [opacity=.5,dotted] (-1,2.25) -- (7,2.25);
      \draw [opacity=.5,dotted] (-1,3.75) -- (7,3.75);
      \draw [opacity=.5, orange, dashed] (-1,.75) -- (1,-.75);
      \node [orange,opacity=.5] at (-1.2,.85) {$0$};
      \draw [opacity=.5, orange, dashed] (-1,2.25) -- (3,-.75);
      \node [orange,opacity=.5] at (-1.2,2.35) {$1$};
      \draw [opacity=.5, orange, dashed] (-1,3.75) -- (5,-.75);
      \node [orange,opacity=.5] at (-1.2,3.85) {$2$};
      \draw [opacity=.5, orange, dashed] (1,3.75) -- (7,-.75);
      \node [orange,opacity=.5] at (1,3.95) {$3$};
      \draw [opacity=.5, orange, dashed] (3,3.75) -- (7,.75);
      \node [orange,opacity=.5] at (3,3.95) {$4$};
      \draw [opacity=.5, orange, dashed] (5,3.75) -- (7,2.25);
      \node [orange,opacity=.5] at (5,3.95) {$5$};
       \node [red] (a) at (0,0) {$P_2$};
      \node [red] (b) at (2,0) {$P_3\langle -1\rangle$};
      \node [red] (c) at (4,0) {$P_4\langle -2\rangle$};
      \node [blue] (d) at (4,1) {$P_4\langle -4\rangle$};
      \node [green] (e) at (4,3) {$P_4\langle -6\rangle$};
      \node [purple] (f) at (4,2) {$P_4\langle -4\rangle$};
      \node [purple] (g) at (2,1.5) {$P_3\langle -3\rangle$};
      \node [purple] (h) at (4,1.5) {$P_2\langle -4\rangle$};
      \node [purple] (i) at (6,1.5) {$P_1\langle -5\rangle$};
      \draw [->] (a) -- (b);
      \draw [->] (b) -- (c);
      \draw [->] (c) -- (d);
      \draw [->] (d) to [bend right] (e);
      \draw [->] (f) -- (e);
      \draw [->] (g) -- (f);
      \draw [->] (g) -- (h);
      \draw [->] (h) -- (i);
    \end{tikzpicture}
    \]
\end{example}

In type $A_2$, we get strong constraints for spherical objects.

\begin{lemma}
In type $A_2$, spherical objects only occupy two columns at most.  
\end{lemma}

\begin{proof}
  In type $A_2$, linear pieces are $P_1$, $P_2$, $P_1\rightarrow P_2\langle -1\rangle$ and $P_2\rightarrow P_1\langle -1\rangle$. For an indecomposable module to occupy more than two columns, both $P_1\rightarrow P_2\langle-1\rangle$ and $P_2\rightarrow P_1\langle -1\rangle$ need to appear. One can see on a picture that this not possible for a curve with both endpoints on a puncture.
\end{proof}

One can read the faithfulness of the 3-strand Burau representation from this, giving a complicated proof to an old known result.

\begin{theorem} \label{thm:3BurauFaithful}
  The 3-strand Burau representation is faithful.
\end{theorem}
\begin{proof}
  In a complex presented as a grid, cancellations in the Grothendieck group occur through side knight moves: those take more than two columns. See the picture below, where we have indicated $q$-degree and homological degree and highlighted two terms that cancel when setting $t=-1$.
\[
\begin{tikzpicture}[anchorbase]
  \node at (-1,1) {$q^{-1}$};
  \node at (-1,0) {$t^{-1}q^{}$};
  \node at (-0,1) {$tq^{-2}$};
  \node at (-0,0) {$t\cdot1$};
  \node at (1,1) {$t^2q^{-3}$};
  \node at (1,0) {$tq^{-1}$};
  \node at (2,1) {$t^3q^{-4}$};
  \node at (2,0) {$t^2q^{-2}$};
  \node at (3,1) {$t^3q^{-5}$};
  \node at (3,0) {$t^3q^{-3}$};
  \draw [orange, thick] (0,1) circle (.45);
  \draw [orange, thick] (2,0) circle (.45);
\end{tikzpicture}
\]

\end{proof}

Unfortunately, the same proof does not extend to the $4$-strand braid group: not only can complexes spread on arbitrarily many columns, but one can even witness cancellations when passing from a spherical complex to its class in the Grothendieck group. This would for example happen for the following object:
\[
(\sigma_1\sigma_2^{-1}\sigma_1^{-1})^{-5}\sigma_3^2\sigma_2^{-1}\sigma_1\sigma_3^{-2}\sigma_2^{-1}P_1.
\]
Below is the associated curve with the two cancelling places marked with $\circ$:
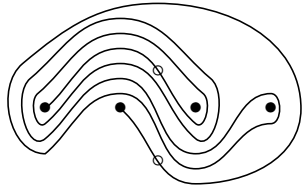
\[
\begin{tikzpicture}[anchorbase]
  \node (1) at (1,0) {\textbullet};
  \node (2) at (2,0) {\textbullet};
  \node (3) at (3,0) {\textbullet};
  \node (4) at (4,0) {\textbullet};
  \draw [semithick] (1.center) to [out=40,in=180] (2,.8) to [out=0,in=140] (3,-.2) to [out=-40,in=-40] (3.1,.2) to [out=140,in=0] (2,1) to [out=180,in=40] (.9,.2) to [out=-140,in=-140] (1,-.2) to [out=40,in=180] (2,.6) to [out=0,in=140] (3,-.4) to [out=-40,in=-40] (3.2,.4) to [out=140,in=0] (2,1.2) to [out=180,in=40] (.8,.4) to [out=-140,in=-140] (1,-.4) to [out=40,in=180] (2,.4) to [out=0,in=180] (3,-.6) to [out=0,in=180] (4,.2) to [out=0,in=0] (4,-.2) to [out=180,in=0] (3,-.8) to [out=180,in=0] (2,.2) to [out=180,in=40] (1,-.6) to [out=180,in=-140] (.7,.6) to [out=40,in=180] (2.5,1.4) to [out=0,in=90] (4.4,0) to [out=-90,in=0] (3,-1) to [out=180,in=-40] (2.center);
  \node at (2.5,-.7) {$\circ$};
  \node at (2.5,.5) {$\circ$};
\end{tikzpicture}
\]

\section{Stability conditions and automata}

\subsection{Bridgeland stability conditions}

A stability condition is a machinery introduced by Bridgeland~\cite{Bridgeland} (in this context) that equips a triangulated category and provides with a way to cut objects in a convenient way.
\begin{definition}
  Let $\mathcal{T}$ be a triangulated category. A stability condition $\tau$ is a pair $(\mathcal{P},Z)$ with:
  \begin{itemize}
  \item $\mathcal{P}$ a slicing, that is, an additive subcategory $\mathcal{P}(\phi)$ for all $\phi\in \R$; objects of these subcategories are called \emph{semi-stable};
  \item $Z$ a central charge, which is a function: $Z:K_0(\mathcal{T})\mapsto \C$.
  \end{itemize}
  $\mathcal{P}$ and $Z$ are subject to the following conditions:
  \begin{enumerate}
  \item $\mathcal{P}(\phi+1)=\mathcal{P}(\phi)[1]$;
  \item if $A\in \calP(\phi)$ and $B\in \calP(\psi)$ with $\phi>\psi$ then $\Hom(A,B)=0$;
  \item for every non-zero object $X\in \mathcal{T}$, there exists a filtration (called \emph{Harder-Narasimhan filtration}):
    \[
    \begin{tikzpicture}
      \node (T0) at (0,0) {$0=X_0$};
      \node (T1) at (2,0) {$X_1$};
      \node (T2) at (4,0) {$X_2$};
      \node (Tdots) at (5,0) {$\cdots$};
      \node (Tn-1) at (6,0) {$X_{n-1}$};
      \node (Tn) at (8,0) {$X_n=X$};
      \node (B1) at (1,-2) {$Y_1$};
      \node (B2) at (3,-2) {$Y_2$};
      \node (Bdots) at (5,-2) {$\cdots$};
      \node (Bn) at (7,-2) {$Y_n$};
      \draw [->] (T0) -- (T1);
      \draw [->] (T1) -- (T2);
      \draw [->] (Tn-1) -- (Tn);
      \draw [->] (T1) -- (B1);
      \draw [->] (T2) -- (B2);
      \draw [->] (Tn) -- (Bn);
      \draw [dashed, ->] (B1) -- (T0) node [midway,xshift=-.2cm,yshift=-.1cm] {\tiny $+1$};
      \draw [dashed, ->] (B2) -- (T1) node [midway,xshift=-.2cm,yshift=-.1cm] {\tiny $+1$};
      \draw [dashed, ->] (Bn) -- (Tn-1) node [midway,xshift=-.2cm,yshift=-.1cm] {\tiny $+1$};
    \end{tikzpicture}
    \]
    where all triangles are exact, $Y_i\in \calP(\phi_i)$ and $\phi_1>\phi_2>\dots>\phi_n$;
  \item $Z\left(\calP(\phi)\setminus \{0\}\right)\subset \R^{+,*}e^{i\pi \phi}$.
  \end{enumerate}
  \end{definition}

In type $A$, stability conditions on $\calC$ have a nice curve-like interpretation. Rather than having all points on a line as in Example~\ref{ex:curvetocomplex}, let us draw them as the vertices of a polygon, as follows:
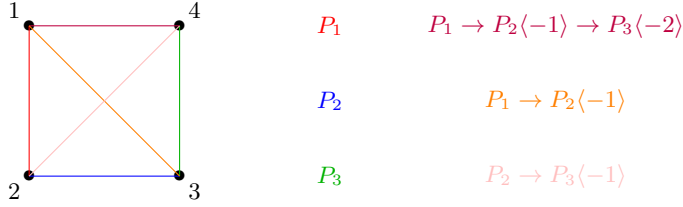
\[
\begin{tikzpicture}[anchorbase]
  \node (1) at (0,2) {$\bullet$};
  \node (2) at (0,0) {$\bullet$};
  \node (3) at (2,0) {$\bullet$};
  \node (4) at (2,2) {$\bullet$};
  \node at (-.2,2.2) {\small $1$}; 
  \node at (-.2,-.2) {\small $2$};
  \node at (2.2,-.2) {\small $3$};
  \node at (2.2,2.2) {\small $4$};
  \draw [red] (1.center) -- (2.center);
  \draw [blue] (2.center) -- (3.center);
  \draw [green] (3.center) -- (4.center);
  \draw [purple] (4.center) -- (1.center);
  \draw [orange] (1.center) -- (3.center);
  \draw [pink] (2.center) -- (4.center);
  \node [red] at (4,2) {\small $P_1$};
  \node [blue] at (4,1) {\small $P_2$};
  \node [green] at (4,0) {\small $P_3$};
  \node [purple] at (7,2) {\small $P_1\rightarrow P_2\langle -1\rangle \rightarrow P_3\langle -2\rangle$};
  \node [orange] at (7,1) {\small $P_1\rightarrow P_2\langle -1 \rangle$};
  \node [pink] at (7,0) {\small $P_2\rightarrow P_3\langle -1\rangle$};
\end{tikzpicture}
\]

Straight line segments correspond to stable objects of the category. The (or rather, a) corresponding stability condition has:
\begin{itemize}
\item $\calP(\phi)=0$ unless $\phi\in \{0,\frac{1}{8},\frac{1}{4},\frac{3}{8},\frac{1}{2},\frac{3}{4}\}+\Z$
\item $\calP(0)=\langle P_1\rangle$,
  $\calP(\frac{1}{8})=\langle P_1 \mapsto P_2\langle -1\rangle\rangle$,
  $\calP(\frac{1}{4})=\langle P_2,P_1{\mapsto}P_2{\langle}-1\rangle\mapsto P_2{\langle}-2\rangle\rangle$\\
  $\calP(\frac{3}{8})=\langle P_2\mapsto P_3\langle -1\rangle \rangle$, $\calP(\frac{1}{2})=\langle P_3\rangle$
\item $|Z(P_1)|=|Z(P_2)|=|Z(P_3)|=1$, the other ones can be deduced easily.
\end{itemize}

It can be convenient to draw the images of all stables in the heart of the $t$-structure (that is, $\calP([0,1[)$). The picture below corresponds to the stability condition described above.
    
    
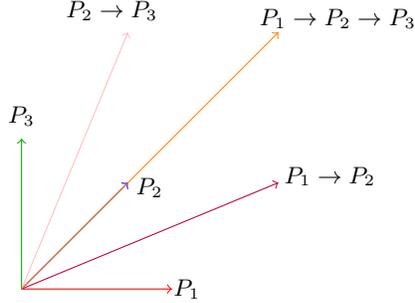
\[
    \begin{tikzpicture}[anchorbase]
      \draw [red,->] (0,0) -- (2,0);
      \node at (2.2,0) {\small $P_1$};
      \draw [purple,->] (0,0) -- (3.414,1.414);
      \node at (4.1,1.5) {\small $P_1\rightarrow P_2$};
      \draw [blue,->] (0,0) -- (1.414,1.414);
      \node at (1.7,1.35) {\small $P_2$};
      \draw [orange,->] (0,0) -- (3.414,3.414);
      \node at (4.2,3.6) {\small $P_1\rightarrow P_2\rightarrow P_3$};
      \draw [pink,->] (0,0) -- (1.414,3.414);
      \node at (1.2,3.7) {\small $P_2\rightarrow P_3$};
      \draw [green,->] (0,0) -- (0,2);
      \node at (0,2.3) {\small $P_3$};
    \end{tikzpicture}
    \]

    Notice that in this case, the slicing is the same as the one we got from the orientation grading. However there are more stability conditions than orientations. For example, one can let $P_1$ and $P_2$ cross, which yields the following:
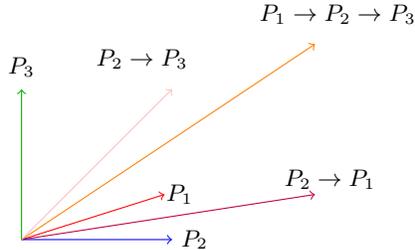
\[
\begin{tikzpicture}[anchorbase]
  \draw [red,->] (0,0) -- (1.9,.6);
  \node at (2.1,.6) {\small $P_1$};
  \draw [purple,->] (0,0) -- (3.9,.6);
  \node at (4.1,.8) {\small $P_2\rightarrow P_1$};
  \draw [blue,->] (0,0) -- (2,0);
  \node at (2.3,0) {\small $P_2$};
  \draw [orange,->] (0,0) -- (3.9,2.6);
  \node at (4.2,3) {\small $P_1\rightarrow P_2\rightarrow P_3$};
  \draw [pink,->] (0,0) -- (2,2);
  \node at (1.6,2.4) {\small $P_2\rightarrow P_3$};
  \draw [green,->] (0,0) -- (0,2);
  \node at (0,2.3) {\small $P_3$};
\end{tikzpicture}
\]

One readily sees that no orientation yields this set of stable objects.

\begin{problem}
  What kind of Garside-like structure or normal form comes from the above stability condition?
\end{problem}

The following theorem of Bridgeland tells us that the moduli space of stability conditions provides us with a nice geometric object, that can be studied for its own sake, but should also be instrumental in the study of Artin-Tits groups, since it inherits an action of the group.

\begin{theorem}[\cite{Bridgeland}]
  $Stab$, the moduli space of stability conditions, is a complex manifold of complex dimension $n$.
\end{theorem}

\subsection{Support}

\begin{definition}
  Given an object $X\in \calC_{\Gamma}$ and a stability condition $\tau$, the \emph{$\tau$-support} of $X$ is the set of semi-stables that appear in its HN-filtration.
\end{definition}

In type $A$, the support can be read from pulling the curve straight.
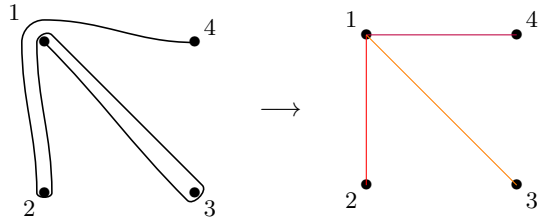
\begin{example}
  \[
  \begin{tikzpicture}[anchorbase]
  \node (1) at (0,2) {$\bullet$};
  \node (2) at (0,0) {$\bullet$};
  \node (3) at (2,0) {$\bullet$};
  \node (4) at (2,2) {$\bullet$};
  \node at (-.4,2.4) {\small $1$}; 
  \node at (-.2,-.2) {\small $2$};
  \node at (2.2,-.2) {\small $3$};
  \node at (2.2,2.2) {\small $4$};
  \draw [semithick] (1.center) to [out=-45,in=135] (1.9,-.1) to [out=-45,in=-45] (2.1,.1) to [out=135,in=-45] (.1,2.1) to [out=135,in=90] (-.1,2) to [out=-90,in=90] (.1,0) to [out=-90,in=-90] (-.1,0) to [out=90,in=-90] (-.3,2) to [out=90,in=180] (0,2.3) to [out=0,in=180] (2,2);
  \end{tikzpicture}
  \quad
  \longrightarrow
  \quad
\begin{tikzpicture}[anchorbase]
  \node (1) at (0,2) {$\bullet$};
  \node (2) at (0,0) {$\bullet$};
  \node (3) at (2,0) {$\bullet$};
  \node (4) at (2,2) {$\bullet$};
  \node at (-.2,2.2) {\small $1$}; 
  \node at (-.2,-.2) {\small $2$};
  \node at (2.2,-.2) {\small $3$};
  \node at (2.2,2.2) {\small $4$};
  \draw [red] (1.center) -- (2.center);
  \draw [purple] (4.center) -- (1.center);
  \draw [orange] (1.center) -- (3.center);
\end{tikzpicture}  
  \]
\end{example}

\begin{problem}
  How do the HN filtration and the curve decomposition play together?
\end{problem}

\subsection{Action of braids on Stab}

Let $\tau \in Stab(\calC_\Gamma)$. Then $\beta \cdot \tau$ for $\beta \in B_{\Gamma}$ is given by $(\beta Z,\beta \calP)$ with :
\[
\begin{cases}
  (\beta \calP)(\phi)=\beta \calP(\phi) \\
  (\beta Z)[X]=Z([\beta^{-1}X])
\end{cases}
\]

So, $Stab$, or $PStab=Stab/(\C^\ast)$, provide us with geometric objects carrying an action of the Artin-Tits group. Furthermore, this action is faithful in all types where the action on $\calC_\Gamma$ is.

\begin{problem}
  Can we use this geometry for anything?
\end{problem}

In type $A_2$, the picture is especially nice: $PStab$ is a variety of real dimension $2$, that admits the following tessellation (see~\cite{BDL} for the origin of this picture):

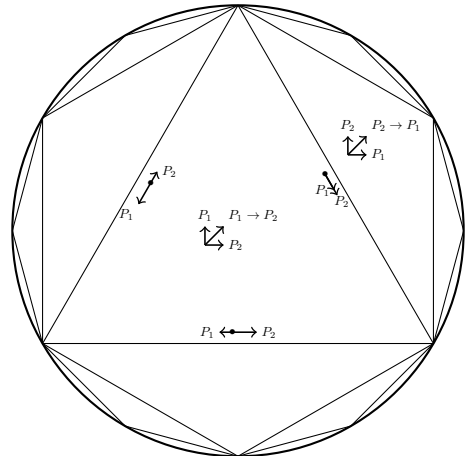
\[
\begin{tikzpicture}[anchorbase]
  \draw [opacity=.4,thick] (0,0) circle (3);
  \draw (90:3) -- (-150:3);
  \draw (90:3) -- (-30:3);
  \draw (-30:3) -- (-150:3);
  \draw (-90:3) -- (-150:3);
  \draw (-90:3) -- (-30:3);
  \draw (30:3) -- (90:3);
  \draw (30:3) -- (-30:3);
  \draw (150:3) -- (-150:3);
  \draw (150:3) -- (90:3);
  \foreach \x in {0,...,11}{
    \draw ({30*\x}:3) -- ({30*\x+30}:3);
  };
  \node at (0,0) {
    \begin{tikzpicture}[anchorbase,scale=.25,every node/.style={scale=0.5}]
      \draw [->,semithick] (0,0) -- (1,0) node [right] {$P_2$};
      \draw [->,semithick] (0,0) --  (0,1) node [above] {$P_1$};
      \draw [->,semithick] (0,0) --  (1,1) node [above right] {$P_1\rightarrow P_2$};
    \end{tikzpicture}
  };
  \node at (1.9,1.2) {
    \begin{tikzpicture}[anchorbase,scale=.25,every node/.style={scale=0.5}]
      \draw [->,semithick] (0,0) -- (1,0) node [right] {$P_1$};
      \draw [->,semithick] (0,0) --  (0,1) node [above] {$P_2$};
      \draw [->,semithick] (0,0) --  (1,1) node [above right] {$P_2\rightarrow P_1$};
    \end{tikzpicture}
  };
  \node at (0,-1.35) {
    \begin{tikzpicture}[anchorbase,scale=.25,every node/.style={scale=0.5}]
      \node at (0,0) {$\bullet$};
      \draw [->,semithick] (0,0) -- (-.7,0) node [left] {$P_1$};
      \draw [->,semithick] (0,0) --  (1.3,0) node [right] {$P_2$};
    \end{tikzpicture}
  };
  \node at (1.25,.55) {
    \begin{tikzpicture}[rotate=-60,anchorbase,scale=.25,every node/.style={scale=0.5}]
      \node at (0,0) {$\bullet$};
      \draw [->,semithick] (0,0) -- (1,0) node [left] {$P_1$};
      \draw [->,semithick] (0,0) --  (1.3,0) node [below right, xshift=-.2cm,yshift=.1cm] {$P_2$};
    \end{tikzpicture}
  };
  \node at (-1.2,.5) {
    \begin{tikzpicture}[rotate=60, anchorbase,scale=.25,every node/.style={scale=0.5}]
      \node at (0,0) {$\bullet$};
      \draw [->,semithick] (0,0) -- (-1.3,0) node [below left] {$P_1$};
      \draw [->,semithick] (0,0) --  (.7,0) node [right] {$P_2$};
    \end{tikzpicture}
  };
\end{tikzpicture}
\]

The similarity of this picture with the Farey triangulation of the hyperbolic disk is no accident, and we will make use of this later on.

\subsection{Automata}

Let $G$ be a group with generators $\{l_i\}$. A finite-state automaton associated to it is:
\begin{itemize}
\item a finite oriented graph (cycles and multi-edges are OK);
\item with edges labeled by generators $l_i$'s;
\item with compatibility with $G$: two paths starting from the same vertex (or state) that represent the same group element should land in the same state.
\end{itemize}
A word is {\it recognized} by the automaton if there is a path that writes that word. A group element is recognized if there is a recognized word for it.

The theory of (bi-)automatic group has been developed in particular by Charney~\cite{Charney_automatic}, and this kind of techniques is useful to solve the word problem.

Bapat, Deopurkar and Licata explain~\cite{BDL} how to build an automaton from the HN support as follows. One builds a graph with vertices labeled by all possible supports of a spherical object. Note that there are finitely many such supports. Then there is a $\beta$-labeled edge between vertices with labels $S_1$ and $S_2$ if for all spherical object $X$ having support $S_1$, $\beta X$ has support $S_2$.

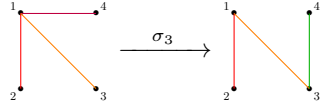
\begin{example}
  \[
  \begin{tikzpicture}[anchorbase,scale=.5, every node/.style={scale=0.5}]
  \node (1) at (0,2) {$\bullet$};
  \node (2) at (0,0) {$\bullet$};
  \node (3) at (2,0) {$\bullet$};
  \node (4) at (2,2) {$\bullet$};
  \node at (-.2,2.2) {\small $1$}; 
  \node at (-.2,-.2) {\small $2$};
  \node at (2.2,-.2) {\small $3$};
  \node at (2.2,2.2) {\small $4$};
  \draw [red] (1.center) -- (2.center);
  \draw [purple] (4.center) -- (1.center);
  \draw [orange] (1.center) -- (3.center);
  \end{tikzpicture}  
  \xrightarrow{\quad \sigma_3 \quad}
  \begin{tikzpicture}[anchorbase,scale=.5, every node/.style={scale=0.5}]
  \node (1) at (0,2) {$\bullet$};
  \node (2) at (0,0) {$\bullet$};
  \node (3) at (2,0) {$\bullet$};
  \node (4) at (2,2) {$\bullet$};
  \node at (-.2,2.2) {\small $1$}; 
  \node at (-.2,-.2) {\small $2$};
  \node at (2.2,-.2) {\small $3$};
  \node at (2.2,2.2) {\small $4$};
  \draw [red] (1.center) -- (2.center);
  \draw [green] (3.center) -- (4.center);
  \draw [orange] (1.center) -- (3.center);
  \end{tikzpicture}    
  \]
  but there is no arrow labeled $\sigma_1$ out of the first vertex.
\end{example}

In type $A_2$ one can easily build an automaton:

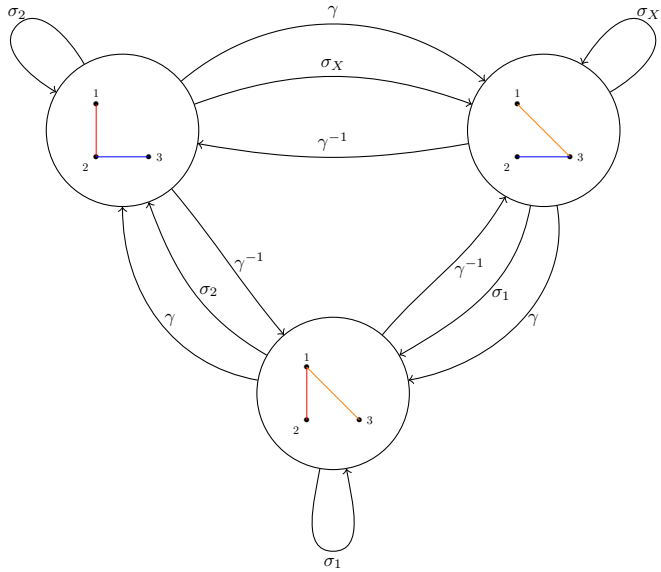
\[
\begin{tikzpicture}[anchorbase, scale=.7, every node/.style={scale=0.7}]
  \node [draw, circle] (A) at (0,0) {
    \begin{tikzpicture}[anchorbase]
      \node at (0,1) {$\bullet$};
      \node at (0,0) {$\bullet$};
      \node at (1,0) {$\bullet$};
      \node at (0,1.2) {\small $1$};
      \node at (-.2,-.2) {\small $2$};
      \node at (1.2,0) {\small $3$};
      \draw [red] (0,1) -- (0,0);
      \draw [blue] (0,0) -- (1,0);
    \end{tikzpicture}
  };
  \node [draw, circle] (B) at (8,0) {
    \begin{tikzpicture}[anchorbase]
      \node at (0,1) {$\bullet$};
      \node at (0,0) {$\bullet$};
      \node at (1,0) {$\bullet$};
      \node at (0,1.2) {\small $1$};
      \node at (-.2,-.2) {\small $2$};
      \node at (1.2,0) {\small $3$};
      \draw [orange] (0,1) -- (1,0);
      \draw [blue] (0,0) -- (1,0);
    \end{tikzpicture}
  };
  \node [draw, circle] (C) at (4,-5) {
    \begin{tikzpicture}[anchorbase]
      \node at (0,1) {$\bullet$};
      \node at (0,0) {$\bullet$};
      \node at (1,0) {$\bullet$};
      \node at (0,1.2) {\small $1$};
      \node at (-.2,-.2) {\small $2$};
      \node at (1.2,0) {\small $3$};
      \draw [red] (0,1) -- (0,0);
      \draw [orange] (0,1) -- (1,0);
    \end{tikzpicture}
  };
  \draw [->] (A) to [out=120,in=45] (-2,2) node [above] {$\sigma_2$} to [out=-135,in=150] (A) ;
  \draw [->] (B) to [out=30,in=-45] (10,2) node [above] {$\sigma_X$} to [out=135,in=60] (B) ;
  \draw [->] (C) to [out=-100,in=180] (4,-8) node [below] {$\sigma_1$} to [out=0,in=-80] (C) ; 
  \draw [->] (A) to [out=20,in=160] node [above] {$\sigma_X$} (B);
  \draw [->] (C) to [out=150,in=-70] node [right] {$\sigma_2$} (A);
  \draw [->] (B) to [out=-100,in=30] node [right] {$\sigma_1$} (C);
  \draw [opacity=.8,->] (A) to [out=40,in=140] node [above] {$\gamma$} (B);
  \draw [opacity=.8,->] (C) to [out=170,in=-90] node [right] {$\gamma$} (A);
  \draw [opacity=.8,->] (B) to [out=-80,in=10] node [right] {$\gamma$} (C);
  \draw [opacity=.8,<-] (A) to [out=-10,in=-170] node [above] {$\gamma^{-1}$} (B);
  \draw [opacity=.8,<-] (C) to [out=130,in=-50] node [right] {$\gamma^{-1}$} (A);
  \draw [opacity=.8,<-] (B) to [out=-120,in=50] node [right] {$\gamma^{-1}$} (C);
\end{tikzpicture}
\]

Above $\sigma_X=\sigma_1\sigma_2\sigma_1^{-1}$.

The way to know if one can have an arrow out of a state is as follows: if one knows the HN filtration of a spherical object $A$, is it true that the image $\sigma_i A$ will always have the same prescribed HN-support? If yes, then the arrow is admissible, if not it shouldn't appear.

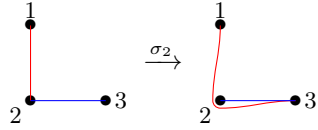
\begin{example}
  \[
   \begin{tikzpicture}[anchorbase]
      \node at (0,1) {$\bullet$};
      \node at (0,0) {$\bullet$};
      \node at (1,0) {$\bullet$};
      \node at (0,1.2) {\small $1$};
      \node at (-.2,-.2) {\small $2$};
      \node at (1.2,0) {\small $3$};
      \draw [red] (0,1) -- (0,0);
      \draw [blue] (0,0) -- (1,0);
   \end{tikzpicture}
   \xrightarrow{\sigma_2}
   \begin{tikzpicture}[anchorbase]
      \node at (0,1) {$\bullet$};
      \node at (0,0) {$\bullet$};
      \node at (1,0) {$\bullet$};
      \node at (0,1.2) {\small $1$};
      \node at (-.2,-.2) {\small $2$};
      \node at (1.2,0) {\small $3$};
      \draw [red] (0,1) to [out=-90,in=90] (-.1,0) to [out=-90,in=180] (0,-.1) to [out=0,in=180] (1,0);
      \draw [blue] (0,0) -- (1,0);
   \end{tikzpicture}  
   \]
   The target has same support. On the other hand applying $\sigma_1$ to the same support can yield different answers:
  
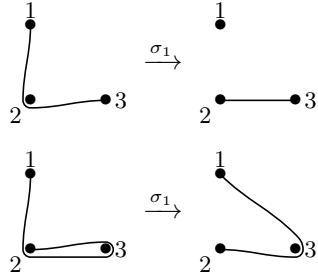
\begin{gather*}
   \begin{tikzpicture}[anchorbase]
      \node at (0,1) {$\bullet$};
      \node at (0,0) {$\bullet$};
      \node at (1,0) {$\bullet$};
      \node at (0,1.2) {\small $1$};
      \node at (-.2,-.2) {\small $2$};
      \node at (1.2,0) {\small $3$};
      \draw [semithick] (0,1) to [out=-90,in=90] (-.1,0) to [out=-90,in=180] (0,-.1) to [out=0,in=180] (1,0);
   \end{tikzpicture}
   \xrightarrow{\sigma_1}
   \begin{tikzpicture}[anchorbase]
      \node at (0,1) {$\bullet$};
      \node at (0,0) {$\bullet$};
      \node at (1,0) {$\bullet$};
      \node at (0,1.2) {\small $1$};
      \node at (-.2,-.2) {\small $2$};
      \node at (1.2,0) {\small $3$};
      \draw [semithick] (0,0) -- (1,0);
   \end{tikzpicture}
   \\
      \begin{tikzpicture}[anchorbase]
      \node at (0,1) {$\bullet$};
      \node at (0,0) {$\bullet$};
      \node at (1,0) {$\bullet$};
      \node at (0,1.2) {\small $1$};
      \node at (-.2,-.2) {\small $2$};
      \node at (1.2,0) {\small $3$};
      \draw [semithick] (0,1) to [out=-90,in=90] (-.1,0) to [out=-90,in=180] (0,-.1) to [out=0,in=180] (1,-.1) to [out=0,in=-90] (1.1,0) to [out=90,in=0] (1,.1) to [out=180,in=0] (0,0);
   \end{tikzpicture}
   \xrightarrow{\sigma_1}
   \begin{tikzpicture}[anchorbase]
      \node at (0,1) {$\bullet$};
      \node at (0,0) {$\bullet$};
      \node at (1,0) {$\bullet$};
      \node at (0,1.2) {\small $1$};
      \node at (-.2,-.2) {\small $2$};
      \node at (1.2,0) {\small $3$};
      \draw [semithick] (0,0) to [out=0,in=180]  (1,-.1) to [out=0,in=-90] (1.1,0) to [out=90,in=-45] (1,.2) to [out=135,in=-45] (0,1);
   \end{tikzpicture}
   \end{gather*}
   
\end{example}

\begin{problem}
  When (and for which set of generators) does such an automaton recognizes all braids?
\end{problem}

The answer in type $A_2$ follows from the following result.

\begin{proposition}[\cite{BDL}]
  Every braid in $\B_{A_2}$ has an expression of the form:
  \[\beta=\gamma^n \sigma_{a_1}^{m_1}\cdots \sigma_{a_k}^{m_k},\]
  with $n$ a relative integer, $k$ a non-negative integer, the $m_i$'s positive integers, and $a_1$, \dots, $a_k$ a contiguous subsequence of $(\dots, X,1,2,X,1,2,X,\dots)$. In particular, every word is recognized by the automaton.
\end{proposition}

\begin{proof}
  Clearly, an expression of this form is recognized by the automaton.

  Let us now define a rewriting process for a general word $w$ in the generators $\sigma_1$, $\sigma_2$ and $\sigma_X$, $w=\sigma_{a_k}\cdots \sigma_{a_1}$. Notice that this set of generators is closed under conjugation by $\gamma$, so powers of $\gamma$ may always slide.
Let us now consider a word $\sigma_{b_k}\cdots \sigma_{b_1}$ (notice that $b_1$ is at the far right, as we'll read from right to left).
  Start from the node with loop $\sigma_{b_1}$. Up to conjugation by a power of $\gamma$, we may assume that $b_1=1$. If $\sigma_{b_2}=\sigma_1$ or $\sigma_{b_2}=\sigma_2$ we are fine, and can conclude by induction. If $\sigma_{b_2}=\sigma_X$, then we notice that $\sigma_X\sigma_1=\sigma_1\sigma_2\sigma_1^{-1}\sigma_1=\sigma_1\sigma_2=\gamma$, that can slide. Then we conclude by induction. 
\end{proof}

This result was used in~\cite{BDL} to study the compactification of $Stab$.

\begin{problem}
  How can one understand the support of a spherical object, especially outside of type $A$?
\end{problem}

\subsection{Haagerup property}

In this section, I'll try to advocate that the tools we've built should be useful to address the Haagerup question. Let's first define this.

\begin{definition}
A finitely generated group $G$ has the Haagerup property if it has a proper action by affine isometries on a Hilbert space.
\end{definition}

Nice consequences of the Haagerup properties are the Baum-Connes property and the Novikov property, and more generally we have a landscape that looks like Figure~\ref{fig:geo}.

\begin{figure}[!h]
  \[
  \begin{tikzpicture}[anchorbase,every node/.style={transform shape}]
    \node [draw, rectangle] (Kazhdan) at (5,2) {Kazhdan's property (T)};
    \node [draw, rectangle] (amenable) at (0,0) {Amenable};
    \node [draw, rectangle] (haagerup) at (0,-1.5) {Haagerup};
    \node [draw, rectangle] (baumconnes) at (-1,-3) {Baum-Connes};
    \node [draw, rectangle] (orderable) at (4.5,-3) {Orderable};
    \node [draw, rectangle] (novikov) at (-1,-4.5) {Novikov};
    \node [draw, rectangle] (kadisonkaplanski) at (4,-4.5) {Kadison-Kaplansky};
    \node [draw, ellipse, fill=red, fill opacity=.6] (wrong) at (-1,2) {Wrong in general};
    \node [draw, ellipse, fill=green, fill opacity=.6] (open) at (7,-1) {Open in general};
    \draw [double,->] (amenable) -- (haagerup);
    \draw [double, ->] (haagerup) -- (baumconnes);
    \draw [ ->] (baumconnes) edge [double] node [midway,xshift=.5 em,rotate=-90] {\tiny disc gps} (novikov);
    \draw [ ->] (baumconnes) edge [double] node [midway,yshift=.3 em,rotate=-16] {\tiny disc gps} (kadisonkaplanski);
    \draw [double, ->] (orderable) -- (kadisonkaplanski);
    \draw [->] (wrong) edge  [bend right] (amenable);
    \draw [->] (open) edge node [midway, yshift=.4em,rotate=5] {\tiny True type $A_2$} (haagerup);
    \draw [->] (open) edge node [midway, yshift=.1em, align=center,rotate=36] {\tiny True type $A$, $D$ \\[-3pt] \tiny special cases} (orderable);
    \draw [->] (open) edge node [midway, yshift=.4em,rotate=12] {\tiny True type A} (baumconnes.north east);
    \draw [->] (wrong) edge [bend left] (Kazhdan);
    \draw [<->] (haagerup) edge [bend right] node [midway,above,rotate=37] {\tiny As far as it gets} (Kazhdan);
  \end{tikzpicture}
\]
\caption{Geometric landscape for Artin-Tits groups} \label{fig:geo}
\end{figure}
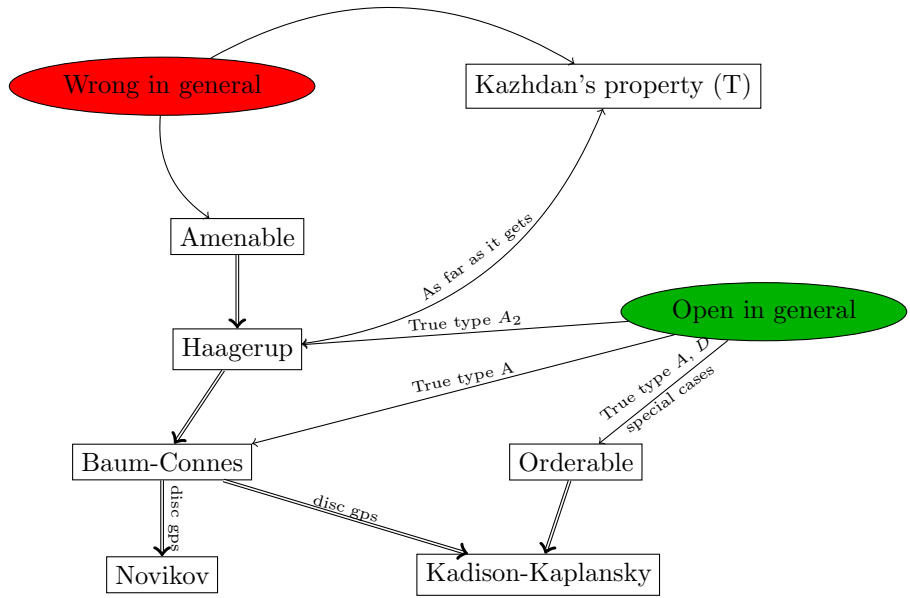

For category theorists that may not be used to work in Hilbert spaces, here is a criterion that we should find more handy.

\begin{lemma}[Haglund-Paulin criterion~\cite{HaglundPaulin}]
  Let $G$ be a finitely generated group, $X$ a set and $\mathcal{H}\subset \mathcal{P}(X)$ a set of so-called walls. Assume that:
  \begin{itemize}
  \item $\forall x,y \in X$, there are finitely many walls separating $x$ and $y$ (that is, finitely many $H$ in $\mathcal{H}$ with $x\in H$, $y \notin H$);
  \item $\forall x$, $\forall (g_n)_{n\in \N}$ with $l(g_n)\rightarrow \infty$, the number of walls separating $x$ and $g_nx$ goes to infinity as $n\rightarrow \infty$.
  \end{itemize}
  Then $G$ has the Haagerup property.
\end{lemma}

One can recover the original Hilbert space by looking at series of walls with the $l^2$ topology.

It has been known for a long time that the 3-strand braid group has the Haagerup property. The classical proof goes by looking at the action of $PSL_3(\Z)=\B_{A_2}/Z(\B_{A_2})$ on the hyperbolic disk. Let's give this the hardest proof ever.

We'll first slightly modify the previous automaton by looking at the support of $\beta\{P_1,P_2,X\}$. Notice that this adds a central state. We'll also add inverses of the generators. In order to keep the picture lighter, I have removed the action of $\gamma$.

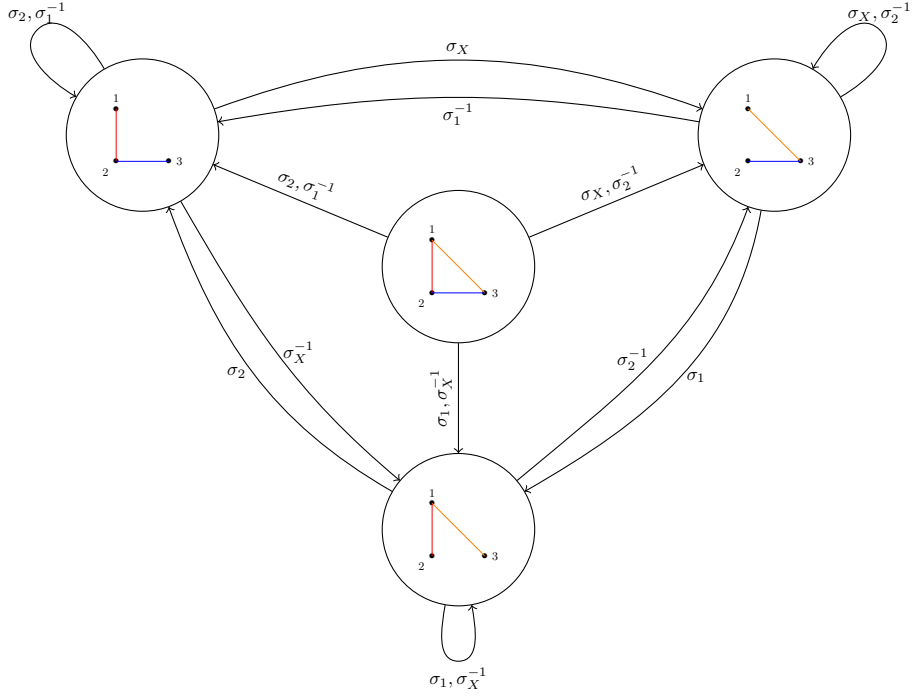
\[
\begin{tikzpicture}[anchorbase, scale=.7, every node/.style={scale=0.7}]
  \node [draw, circle] (A) at (0,0) {
    \begin{tikzpicture}[anchorbase]
      \node at (0,1) {$\bullet$};
      \node at (0,0) {$\bullet$};
      \node at (1,0) {$\bullet$};
      \node at (0,1.2) {\small $1$};
      \node at (-.2,-.2) {\small $2$};
      \node at (1.2,0) {\small $3$};
      \draw [red] (0,1) -- (0,0);
      \draw [blue] (0,0) -- (1,0);
    \end{tikzpicture}
  };
  \node [draw, circle] (B) at (12,0) {
    \begin{tikzpicture}[anchorbase]
      \node at (0,1) {$\bullet$};
      \node at (0,0) {$\bullet$};
      \node at (1,0) {$\bullet$};
      \node at (0,1.2) {\small $1$};
      \node at (-.2,-.2) {\small $2$};
      \node at (1.2,0) {\small $3$};
      \draw [orange] (0,1) -- (1,0);
      \draw [blue] (0,0) -- (1,0);
    \end{tikzpicture}
  };
  \node [draw, circle] (C) at (6,-7.5) {
    \begin{tikzpicture}[anchorbase]
      \node at (0,1) {$\bullet$};
      \node at (0,0) {$\bullet$};
      \node at (1,0) {$\bullet$};
      \node at (0,1.2) {\small $1$};
      \node at (-.2,-.2) {\small $2$};
      \node at (1.2,0) {\small $3$};
      \draw [red] (0,1) -- (0,0);
      \draw [orange] (0,1) -- (1,0);
    \end{tikzpicture}
  };
  \node [draw, circle] (M) at (6,-2.5) {
      \begin{tikzpicture}[anchorbase]
      \node at (0,1) {$\bullet$};
      \node at (0,0) {$\bullet$};
      \node at (1,0) {$\bullet$};
      \node at (0,1.2) {\small $1$};
      \node at (-.2,-.2) {\small $2$};
      \node at (1.2,0) {\small $3$};
      \draw [red] (0,1) -- (0,0);
      \draw [orange] (0,1) -- (1,0);
      \draw [blue] (0,0) -- (1,0);
    \end{tikzpicture}
  };
  \draw [->] (A) to [out=120,in=45] (-2,2) node [above] {$\sigma_2,\sigma_1^{-1}$} to [out=-135,in=150] (A) ;
  \draw [->] (B) to [out=30,in=-45] (14,2) node [above] {$\sigma_X, \sigma_2^{-1}$} to [out=135,in=60] (B) ;
  \draw [->] (C) to [out=-100,in=180] (6,-10) node [below] {$\sigma_1,\sigma_X^{-1}$} to [out=0,in=-80] (C) ; 
  \draw [->] (A) to [out=20,in=160] node [above] {$\sigma_X$} (B);
  \draw [->] (C) to [out=150,in=-70] node [left] {$\sigma_2$} (A);
  \draw [->] (B) to [out=-100,in=30] node [right] {$\sigma_1$} (C);
  \draw [<-] (A) to [out=10,in=170] node [below] {$\sigma_1^{-1}$} (B);
  \draw [<-] (C) to [out=140,in=-60] node [right] {$\sigma_X^{-1}$} (A);
  \draw [<-] (B) to [out=-110,in=40] node [left] {$\sigma_2^{-1}$} (C);
  \draw [->] (M) to node[above,rotate=-20] {$\sigma_2,\sigma_1^{-1}$} (A);
  \draw [->] (M) to node[above,rotate=20] {$\sigma_X,\sigma_2^{-1}$} (B);
  \draw [->] (M) to node[above,rotate=90] {$\sigma_1,\sigma_X^{-1}$} (C);
\end{tikzpicture}
\]

We will define a set of walls on $\mathcal{X}=\B_{A_2}\cdot \{P_1,P_2,X\}$\footnote{Secretly we're acting on $Stab$...} that satisfy the Haglund-Paulin criterion.

The automaton above gives us three walls: for each of the three outer nodes, either $\{A,B,C\}\in \mathcal{X}$ has support on this node, or not. Now we could have done the same with any other stability condition, and we would have had a similar automaton. So consider the set of walls that's induced by all stability conditions $\beta\cdot\tau_0$, if $\tau_0$ was one that had $\mathbf{D}=\{P_1,P_2,X\}$ as stables.

We want to check both conditions. Start with proving that there are finitely many walls between any two points in $\mathcal{X}$. It is enough to check it for $\mathbf{D}$ and $\sigma_2^{-1}\mathbf{D}$. Consider a braid $\mu$, and the stability condition $\tau=\mu^{-1}\tau_0$. It has stables $\mu^{-1}\mathbf{D}$. Then,
\begin{gather*}
  HN_{\mu^{-1}\tau_0}(\mathbf{D})=\mu^{-1}\cdot HN_{\tau_0}(\mu \mathbf{D}) \\
  HN_{\mu^{-1}\tau_0}(\sigma_2^{-1}\mathbf{D})=\mu^{-1}\cdot HN_{\tau_0}(\mu\sigma_2^{-1} \mathbf{D})
\end{gather*}
So the question comes down to: does $\tau_0$ separate (for one of its three walls) $\mu\mathbf{D}$ and $\mu\sigma_2^{-1} \mathbf{D}$? Or in other words: do $\mu \mathbf{D}$ and $\mu\sigma_2^{-1}\mathbf{D}$ belong to the same state?

Let's take for $\mu$ an expression that's recognized from the ground state. Notice that $\sigma_2^{-1}$ takes the ground state to the one at the top right. If $\mu$ is also recognized from this state, then both $\mu$ and $\mu \sigma_2^{-1}$ have the same target state, as a given letter always maps to the same state. Then $\tau_0$ does not tell the two points apart.

Now let's assume that $\mu$ is not recognized from the top right state. Write $\mu=\mu_k\cdots \mu_1$ (recognized from the ground state). Notice that since $\mu_1$ maps to the same place from any state, $\mu_1$ can't be recognized from the top right state, otherwise $\mu$ would be. But that means that $\mu_1=\sigma_2$ or $\mu_1=\sigma_X^{-1}$. If $k=1$, then we see three obvious walls telling apart $\mathbf{D}$ and $\sigma_2^{-1}\mathbf{D}$. If $k>1$, then the question gets translated to $\tau_0$ telling apart $(\mu_k\cdots \mu_2 \sigma_2 \mathbf{D},\mu_k\cdots \mu_2\mathbf{D})$ (if $\mu_1=\sigma_2$) or $(\mu_k \cdots \mu_2\sigma_X\mathbf{D},\mu_k\cdots \mu_2\mathbf{D})$ (if $\mu_1=\sigma_X^{-1}$ and using $\sigma_X^{-1}\sigma_2^{-1}=\gamma^{-1}$).

Consider the first case. We know that $\mu_k\cdots \mu_2$ is recognized from the top left state, so also from the central state (this is a key property of the automaton that's easily established in this case). So now both $\mu_k\cdots \mu_2$ and $\mu_k\cdots \mu_2\sigma_2$ are recognized from the central state, which means that the have the same target state, and the two elements of $\mathcal{X}$ are not told apart.

The second case is similar: $\mu_k\cdots \mu_2$ is recognized from the bottom state, so also from the central state, and thus $\mu_k\cdots \mu_2$ and $\mu_k\cdots \mu_2\cdot \sigma_X^{-1}$ have same target space.

Let us now consider the statement about walls to infinity. Consider a braid $\beta_n\cdot \beta_1$ in the letters $\sigma_1$, $\sigma_2$ and $\sigma_X$, recognized by the automaton from the ground state.

\noindent {\bf Claim:} $\beta_n\cdots \beta_k\tau_0$ tells apart $\mathbf{D}$ and $\beta\mathbf{D}$ for all $k$.

This amounts to comparing the $\tau_0$-supports of $\beta_k^{-1}\cdots \beta_n^{-1}\mathbf{D}$ and $\beta_{k-1}\cdots \beta_1\mathbf{D}$.

Notice the following property of this automaton: if $\mu_n\cdots \mu_1$ is recognized, then so is $\mu_1^{-1}\cdots \mu_n^{-1}$ (look at admissible 2-words).

This means that the right subword $\beta_k^{-1}\cdots \beta_n^{-1}$ is recognized from the ground state, ending at the state with loop $\beta_k^{-1}$. Similarly, $\beta_{k-1}\cdots \beta_1$ is recognized from the ground state, ending at the state with loop $\beta_{k-1}$. These two states are different: one can check that if $u$ and $v$ are loops at the same state, then $uv^{-1}$ is never recognized.
 
So there indeed are infinitely many walls.

\appendix
\section{Quadratic dual} \label{appendix:quadratic_dual}

The point of this appendix is to clarify the relation between Khovanov-Seidel's zigzag algebra and the preprojective algebra~\cite{GelfandPonomarev}, as mentioned for example in~\cite{QiSussan}.

Classically (thank you Wikipedia), a quadratic algebra is a filtered algebra generated by elements in degree one, subject to homogeneous relations of degree 2. Here we need to slightly modify the definition, as we also have idempotents. So let's say that a non-unital quadratic algebra is an algebra with given idempotents in degree zero, other generators in degree one, and homogeneous relations of degree 2 (that come in addition to the idempotency ones).

The classical picture is that of an algebra presented as $T(V)/\langle S\rangle$, with $V$ a vector space generated by the generators and $S\subset V\otimes V$ the vector space generated by the relations. Then the quadratic dual is the quadratic algebra generated by $V^\ast$ with relations $S^{\perp}\subset V^\ast \otimes V^\ast$.

Let's start from the preprojective algebra in type $A_n$, with orthogonal idempotents $e_i$'s, generators $\gamma_{i,i+1}=(i|i+1)=e_i\gamma_{i,i+1}=\gamma_{i,i+1}e_{i+1}$ and $\gamma_{i,i-1}=(i|i-1)=e_i\gamma_{i,i-1}=\gamma_{i,i-1}e_{i-1}$, and relations $(i|i+1|i)=(i|i-1|i)$.

Let's consider an idempotented version of the quadratic dual by considering within $V\otimes V$ the span of composable pairs: $V\otimes_{\oplus \langle e_i\rangle}V$. This space has dimension $4n-6$: it is generated by pairs $\gamma_{i\pm 1,i}\otimes \gamma_{i,i\pm 1}$, which makes 4 generators per node except for the extremal ones that contribute only one each. Let's introduce $\{p_i\}$ a set of orthogonal idempotents (duals of the $e_i$'s), and consider the algebra generated by the $p_i$'s and $p_i\gamma_{i+1,i}^\ast p_{i+1}=\gamma_{i+1,i}^\ast$, with relations $S^{\perp}\subset V\otimes_{\oplus \langle e_i\rangle}V$. Since $S$ generates a space of dimension $2n-2$ in a space of dimension $4n-6$, we expect $3n-4$ independent relations. $3n-6$ of them are clear:
\begin{itemize}
\item $\gamma_{i,i+1}^\ast\otimes \gamma_{i+1,i}^\ast+\gamma_{i,i-1}^\ast\otimes \gamma_{i-1,i}^\ast$;
\item $\gamma_{i,i+1}^\ast\otimes \gamma_{i+1,i+2}^\ast$;
\item $\gamma_{i,i-1}^\ast \otimes \gamma_{i-1,i-2}^\ast$.
\end{itemize}
By considering the endpoints, one also finds $\gamma_{1,2}^\ast\otimes \gamma_{2,1}^\ast$ and $\gamma_{n,n-1}^\ast\otimes \gamma_{n-1,n}^\ast$. In order to have matching signs, one can define: $x_{i,i+1}=(-1)^i\gamma_{i,i+1}^\ast$.

We finally find that the idempotented quadratic dual of the preprojective algebra is a twice-based version of the Khovanov-Seidel zigzag algebra:
\begin{itemize}
\item generators $p_i$ and $x_{i,i+1}=p_ix_{i,i\pm1}p_{i\pm1}$;
\item relations:
  \begin{itemize}
  \item idempotency,
  \item $x_{i,i\pm 1}x_{i\pm 1,i}=x_{i,i\mp 1}x_{i\mp 1,i}$,
  \item $x_{i,i\pm 1}x_{i\pm 1,i\pm 2}=0$,
  \item $x_{1,2}x_{2,1}=0$, $x_{n,n-1}x_{n-1,n}=0$.
  \end{itemize}
\end{itemize}


\providecommand{\bysame}{\leavevmode\hbox to3em{\hrulefill}\thinspace}
\providecommand{\MR}{\relax\ifhmode\unskip\space\fi MR }
\providecommand{\MRhref}[2]{%
  \href{http://www.ams.org/mathscinet-getitem?mr=#1}{#2}
}
\providecommand{\href}[2]{#2}

%
\end{document}